\documentclass[11pt,reqno]{amsart}
\usepackage{amsmath,amssymb}
\usepackage[colorlinks=true,linkcolor=magenta,citecolor=blue]{hyperref}
\usepackage[margin=1in]{geometry}
\usepackage{enumitem} 
\usepackage{mathabx}
\usepackage{epsfig}

\newtheorem{dn}{Definition}[section]
\newtheorem{dl}{Theorem}[section]
\newtheorem{md}{Proposition}[section]
\newtheorem{bd}{Lemma}[section]
\newtheorem{hq}{Corollary}[section]
\newtheorem{nx}{Remark}[section]
\newtheorem{vd}{Example}[section]
\newcommand{\R}{\mathbb{R}}

\newcommand{\Z}{\mathbb{Z}}
\newcommand{\N}{\mathbb{N}}
\newcommand{\e}{\varepsilon}
\newcommand{\ity}{\infty}

\newcommand{\bbd}{\begin{bd}}
\newcommand{\ebd}{\end{bd}}
\newcommand{\bdn}{\begin{dn}}
\newcommand{\edn}{\end{dn}}
\newcommand{\bhq}{\begin{hq}}
\newcommand{\ehq}{\end{hq}}
\newcommand{\bdl}{\begin{dl}}
\newcommand{\edl}{\end{dl}}
\newcommand{\bnx}{\begin{nx}}
\newcommand{\enx}{\end{nx}}
\newcommand{\bmd}{\begin{md}}
\newcommand{\emd}{\end{md}}
\newcommand{\bvd}{\begin{vd}}
\newcommand{\evd}{\end{vd}}

\title[Semi-linear $\sigma$-evolution equations with frictional and visco-elastic damping]{Study of semi-linear $\sigma$-evolution equations with frictional and visco-elastic damping}
\author{Hironori Michihisa}
\address{Hironori Michihisa \hfill\break
Department of Mathematics, Graduate School of Science, Hiroshima University \hfill\break
Higashi-Hiroshima 739-8526, Japan}
\email{hi.michihisa@gmail.com}

\author{Tuan Anh Dao$^*$}
\address{Tuan Anh Dao \hfill\break
School of Applied Mathematics and Informatics, Hanoi University of Science and Technology, No.1 Dai Co Viet road, Hanoi, Vietnam \hfill\break
Faculty for Mathematics and Computer Science, TU Bergakademie Freiberg, Pr\"{u}ferstr. 9, 09596, Freiberg, Germany}
\email{anh.daotuan@hust.edu.vn}
\thanks{$^*$ Corresponding author: Tuan Anh Dao}

\begin{document}

\subjclass[2010]{Primary: 35G25, 35B40; Secondary: 35B33, 35C20.}
\keywords{$\sigma$-evolution equations; frictional damping; visco-elastic damping; asymptotic profile; diffusion phenomenona; global existence; critical exponent}
	
\begin{abstract}
In this article, we study semi-linear $\sigma$-evolution equations with double damping including frictional and visco-elastic damping for any $\sigma\ge 1$. We are interested in investigating not only higher order asymptotic expansions of solutions but also diffusion phenomenon in the $L^p-L^q$ framework, with $1\le p\le q\le \ity$, to the corresponding linear equations. By assuming additional $L^{m}$ regularity on the initial data, with $m\in [1,2)$, we prove the global (in time) existence of small data energy solutions and indicate the large time behavior of the global obtained solutions as well to semi-linear equations. Moreover, we also determine the so-called critical exponent when $\sigma$ is integers.
\end{abstract}
\maketitle

\tableofcontents

\section{Introduction and main results}
In this paper, let us consider the following Cauchy problem for semi-linear $\sigma$-evolution equations with frictional and visco-elastic damping terms:
\begin{equation}
\begin{cases}
u_{tt}+ (-\Delta)^\sigma u+ u_t+ (-\Delta)^{\sigma} u_t= |u|^p, \\
u(0,x)= u_0(x),\quad u_t(0,x)=u_1(x),
\end{cases}
\label{equation1.1}
\end{equation}
where $\sigma \ge 1$ and a given real number $p>1$. The corresponding linear equation with vanishing right-hand side is
\begin{equation}
\begin{cases}
u_{tt}+ (-\Delta)^\sigma u+ u_t+ (-\Delta)^{\sigma} u_t= 0, \\
u(0,x)= u_0(x),\quad u_t(0,x)=u_1(x).
\end{cases}
\label{equation1.2}
\end{equation}
At first, let us recall some recent results concerning the study of typical important problems of (\ref{equation1.1}) and (\ref{equation1.2}) with $\sigma=1$, the so-called wave equations with frictional damping and visco-elastic damping. Of special interest are the following Cauchy problems:
\begin{equation}
\begin{cases}
u_{tt}- \Delta u+ u_t- \Delta u_t= |u|^p, \\
u(0,x)= u_0(x),\quad u_t(0,x)=u_1(x).
\end{cases}
\label{equation1.3}
\end{equation}
Namely, in \cite{IkehataSawada} the authors derived the asymptotic profile of the solutions in the $L^2$ setting to the corresponding linear equations of (\ref{equation1.3}) by assuming weighted $L^{1,1}$ initial data from the energy space. In comparison with the two types of damping terms, they analyzed the interesting properties which tell us that the effect of the frictional damping is really more dominant than that of the visco-elastic one, the so-called strong damping (see, for example, \cite{Ikehata,IkehataTodorovaYordanov}), by the study of asymptotic profile as $t\to \ity$. In addition, the higher order (up to the first order) asymptotic profiles of the solutions to the linear corresponding equations of (\ref{equation1.3}) were discussed in the space dimension $n=1$ only. Quite recently, the authors in \cite{IkehataMichihisa} have succeeded in obtaining some higher order (greater than the second order) asymptotic expansions of the solutions to this linear equation under more heavy moment conditions on the initial data for any space dimensions by applying Taylor expansion theorem effectively (see more \cite{Michihisa1,Michihisa2,Takeda}). For the treatment of the semi-linear equations (\ref{equation1.3}), in \cite{Dabbicco} some obtained energy estimates combined with $L^1-L^1$ estimates come into play to prove the global (in time) solutions for any space dimensions. Moreover, taking into consideration the effect of the two damping types as mentioned in \cite{IkehataSawada} to the corresponding linear problem the authors in \cite{IkehataTakeda2017} pointed out again this effect which is still true for the semi-linear problems (\ref{equation1.3}). In particular, they indicated that the critical exponent $p_{crit}=1+\frac{2}{n}$ coincides with the so-called Fujita exponent which is well-known to be the critical exponent for the semi-linear heat equations and the semi-linear classical damped wave equations as well with nonlinearity term $|u|^p$. Besides, not only the existence of the global solutions to (\ref{equation1.3}) has been investigated but also the large time behavior of the obtained global solutions has been established in low space dimensions in \cite{IkehataTakeda2017}.
\par Hence, related to the more general cases of (\ref{equation1.1}) and (\ref{equation1.2}) with $\sigma \ge 1$, a natural question is whether or not the frictional damping is still more dominant than the visco-elastic one for any $\sigma\ge 1$ as it happened for the case $\sigma=1$. One of the main goals of this paper is to give a positive answer to this question. More recently, the authors in \cite{DaoReissig1,DaoReissig2} have studied the following Cauchy problem for structurally damped $\sigma$-evolution equations (see also \cite{DabbiccoEbert2016,DabbiccoEbert2017,DuongReissig}):
\begin{equation}
\begin{cases}
u_{tt}+ (-\Delta)^\sigma u+ (-\Delta)^{\delta} u_t= 0, \\
u(0,x)= u_0(x),\quad u_t(0,x)=u_1(x).
\end{cases}
\label{equation1.4}
\end{equation}
From the point of view of decay estimates, they emphasized that the properties of the solutions change completely from the case $\delta=0$, corresponding to the frictional damping, to the case $\delta=\sigma$, corresponding to the visco-elastic damping. More in detail, they proposed to distinguish between ``parabolic like models" in the former case (see \cite{DabbiccoReissig,ReissigEbert} for the classical damped wave equations in the case $\sigma=1$) and ``$\sigma$-evolution like models" in the latter case, the so-called ``hyperbolic like models" or ``wave like models" in the case $\sigma=1$ (see \cite{DabbiccoReissig,GalaktionovMitidieri}). Roughly speaking, the asymptotic profile of the solutions to (\ref{equation1.4}) with $\delta=0$, as $t \to \ity$, is same as that of the following anomalous diffusion equations:
\begin{equation}
v_t+ (-\Delta)^\sigma v= 0, \qquad v(0,x)= v_0(x),
\label{equation1.5}
\end{equation}
for a suitable choice of data $v_0$ (see, for instance, \cite{DabbiccoEbert2016,DuongReissig}). Meanwhile, for the case $\delta=\sigma$ this phenomenon is no longer true, that is, some kind of wave structure appears and oscillations come into play from the asymptotic profile of the solutions to (\ref{equation1.4}). Furthermore, compared with the regularity of the initial data we can see that a smoothing effect appears for some derivatives of the solutions to (\ref{equation1.4}) with respect to the time variable (see \cite{DaoReissig2}) in the latter case. This brings some benefits in treament of the corresponding semi-linear equations. Otherwise, in the former case this effect does not happen (see \cite{DaoReissig1}). In the connection between the two types of damping terms appearing in (\ref{equation1.2}), the asymptotic profile of the solutions inherits both these above mentioned properties of the two kind of models to give new results. For this reason, the second main goal of the present paper is to conclude a diffusion phenomenon not only in the $L^2-L^2$ theory (see more \cite{DuongReissig,ReissigEbert}) but also in the $L^p-L^q$ framework (see also \cite{DabbiccoEbert2014,Narazaki}), where $1\le p\le q\le \ity$. Moreover, we also establish some higher order asymptotic expansions of the difference between the solutions to (\ref{equation1.2}) and those to (\ref{equation1.5}) by developing several techniques in \cite{IkehataMichihisa}. In order to explain these results more precisely, one knows that these results come from estimates for small-frequency part of the solutions to (\ref{equation1.2}) whose profile is modified by the presence of the fractional damping, whereas their large-frequency profile is modified by the presence of the visco-elastic damping. Our third main goal of this paper is to prove the global (in time) existence of small data energy solutions to (\ref{equation1.1}) and analyze the large time behavior of these global solutions as well by mixing additional $L^{m}$ regularity for the data with $m\in [1,2)$. Finally, when $\sigma$ is integers, a blow-up result is shown to find the critical exponent $p_{crit}=1+\frac{2\sigma}{n}$.

\subsection{Notations}
\par Throughout the present paper, we use the following notations.
\begin{itemize}[leftmargin=*]
\item We write $f\lesssim g$ when there exists a constant $C>0$ such that $f\le Cg$, and $f \approx g$ when $g\lesssim f\lesssim g$.
\item We denote $\hat{f}(t,\xi):= \mathcal{F}_{x\rightarrow \xi}\big(f(t,x)\big)$ as the Fourier transform with respect to the space variable of a function $f(t,x)$. As usual, $H^{a}$ and $\dot{H}^{a}$, with $a \ge 0$, denote Bessel and Riesz potential spaces based on $L^2$ spaces. Here $\big<D\big>^{a}$ and $|D|^{a}$ stand for the pseudo-differential operators with symbols $\big<\xi\big>^{a}$ and $|\xi|^{a}$, respectively.
\item For any $\gamma>0$, the weighted spaces $L^{1,\gamma}(\R^n)$ are defined by
$$ L^{1,\gamma}(\R^n):= \Big\{f\in L^1(\R^n) \text{ such that } \|f\|_{L^{1,\gamma}}:= \int_{\R^n} (1+|x|)^\gamma f(x)\,dx<+\infty \Big\}. $$
\item For any $s \in \R$, we denote $[s]^+:= \max\{s,0\}$ as its positive part, and $[s]:= \max \big\{k \in \Z \,\, : \,\, k\le s \big\}$ as its integer part.
\item We denote $\N_0:= \N \cup \{0\}$.
\item For later convenience, we put
$$ G_\sigma(t,x):=\mathcal{F}^{-1}\big(e^{-t|\xi|^{2\sigma}}\big)(x), $$
and denote the following two quantities:
$$P_0:= \int_{\R^n}u_0(x)dx \quad \text{ and }\quad P_1:= \int_{\R^n}u_1(x)dx. $$
\item Let $\chi(|\xi|)$ be $\mathcal{C}_0^\ity(\R^n)$ a smooth cut-off function equal to $1$ for small $|\xi|$ and vanishing for large $|\xi|$. We decompose a function $f(t,x)$ into two parts localized separately at low and high frequencies as follows:
$$ f(t,x)= f_{\text{\fontshape{n}\selectfont low}}(t,x)+ f_{\text{\fontshape{n}\selectfont high}}(t,x), $$
where we denote \[f_{\text{\fontshape{n}\selectfont low}}(t,x)= \mathcal{F}^{-1}\big(\chi(|\xi|)\hat{f}(t,\xi)\big)\quad \text{ and }\quad f_{\text{\fontshape{n}\selectfont high}}(t,x)= \mathcal{F}^{-1}\Big(\big(1-\chi(|\xi|)\big)\hat{f}(t,\xi)\Big). \]
\item Applying the Fourier transform to (\ref{equation1.2}) we have
\begin{equation}
\begin{cases}
\hat{u}_{tt}+ |\xi|^{2\sigma} \hat{u}+ \hat{u}_t+|\xi|^{2\sigma} \hat{u}_t= 0, \\
\hat{u}(0,\xi)= \widehat{u_0}(\xi),\quad \hat{u}_t(0,\xi)=\widehat{u_1}(\xi).
\end{cases}
\label{equation2.2}
\end{equation}
The characteristic equation of \eqref{equation2.2} is 
\begin{equation}
\lambda^2+ (1+|\xi|^{2\sigma}) \lambda+ |\xi|^{2\sigma}= 0.
\label{equation2.3}
\end{equation}
The solution to \eqref{equation2.3} can be given by 
\begin{align*}
\lambda_\pm
& =\frac{-(1+|\xi|^{2\sigma})\pm\sqrt{(1+|\xi|^{2\sigma})^2-4|\xi|^{2\sigma}}}{2} 
=\frac{-(1+|\xi|^{2\sigma})\pm\big|1-|\xi|^{2\sigma}\big|}{2},
\end{align*}
i.e., 
\begin{align*}
\begin{cases}
\lambda_+= -|\xi|^{2\sigma}, \quad \lambda_-= -1 &\text{ if } |\xi|\le1, \\
\lambda_+= -1, \quad \lambda_-= -|\xi|^{2\sigma} &\text{ if } |\xi|\ge1.
\end{cases} 
\end{align*}
So we explicitly write down the solution formula in the Fourier space as follows: 
\begin{align}
\hat{u}(t,\xi)
& =\frac{1}{1-|\xi|^{2\sigma}}e^{-t|\xi|^{2\sigma}}\big(\widehat{u_0}(\xi)+\widehat{u_1}(\xi)\big)
-\frac{1}{1-|\xi|^{2\sigma}}e^{-t}\big(|\xi|^{2\sigma} \widehat{u_0}(\xi)+\widehat{u_1}(\xi)\big) \label{equation2.4} \\
& =\frac{e^{-t|\xi|^{2\sigma}}-|\xi|^{2\sigma} e^{-t}}{1-|\xi|^{2\sigma}}\widehat{u_0}(\xi)
+\frac{e^{-t|\xi|^{2\sigma}}-e^{-t}}{1-|\xi|^{2\sigma}}\widehat{u_1}(\xi) \label{equation2.5} \\
& =\frac{1}{|\xi|^{2\sigma}-1}e^{-t}\big(|\xi|^{2\sigma} \widehat{u_0}(\xi)+\widehat{u_1}(\xi)\big)
-\frac{1}{|\xi|^{2\sigma}-1}e^{-t|\xi|^{2\sigma}}\big(\widehat{u_0}(\xi)+\widehat{u_1}(\xi)\big). \label{equation2.6}
\end{align}
Note that $\{\xi \in \R^n : |\xi|=1\}$ is not a singular set. 
Indeed, we can give an equivalent formula: 
\begin{align}
\hat{u}(t,\xi)
=e^{-t}\widehat{u_0}(\xi)
+\left(
e^{-t}
\int_0^t
e^{-s(|\xi|^{2\sigma}-1)} \,
ds
\right)
(\widehat{u_0}(\xi)+\widehat{u_1}(\xi)).
\label{sol}
\end{align}
\item For purpose of this paper, we write $v_0:= u_0+ u_1$. 
So we fix the initial data $v_0= u_0+ u_1$ to (\ref{equation1.5}).
\end{itemize}

\subsection{Main results}
Let us state the main results that will be proved in this paper.
\par At first, we indicate higher order asymptotic expansions of the solutions to (\ref{equation1.2}) with weighted initial data.
\bdl \label{TheoremLinear}
Let $n\ge1$ and $u_0,\,u_1\in L^{1,\gamma}(\R^n)\cap L^2(\R^n)$ with $\gamma\ge0$. 
Take $k\in\N_0$ satisfying $\lambda(k)\le\gamma<\lambda(k+1)$.  
Then, the function $u$ defined by \eqref{equation2.4}-\eqref{sol} satisfies 
\begin{equation}
\label{AE1}
\big\|\hat{u}(t,\xi)-A_k^{\sigma,v_0}(\xi)e^{-t|\xi|^{2\sigma}}\big\|_{L^2}
\lesssim t^{-\frac{n}{4\sigma}-\frac{\gamma}{2\sigma}}\big(\|v_0\|_{L^{1,\gamma}}+\|u_0\|_{L^1\cap L^2}+\|u_1\|_{L^1\cap L^2}\big),
\qquad
t\ge1.
\end{equation}
Here, we introduce $A_k^{\sigma,v_0}(\xi)$ as in the expression (\ref{XXX}) (see Definition \ref{Def2.2}). Furthermore, it holds 
\begin{equation}
\label{AE2}
\lim_{t\to\infty}
t^{\frac{n}{4\sigma}+\frac{\gamma}{2\sigma}}
\big\|\hat{u}(t,\xi)-A_k^{\sigma,v_0}(\xi)e^{-t|\xi|^{2\sigma}}\big\|_{L^2}
=0.
\end{equation}
\edl

The second result is concerned with the diffusion phenomenon in the $L^p-L^q$ framework with $1\le p\le q\le \ity$ to (\ref{equation1.2}).
\bdl \label{theorem2.1}
Let $1\le p\le q\le \ity$. Let $u$ be the solution to (\ref{equation1.2}) and let $v$ be the solution to (\ref{equation1.5}). Then, for large $t \ge 1$ we have the following $L^p-L^q$ estimate:
$$ \big\|\partial_t^j |D|^a \big(u_{\text{\fontshape{n}\selectfont low}}(t,\cdot)- v_{\text{\fontshape{n}\selectfont low}}(t,\cdot)\big)\big\|_{L^q} \lesssim (1+t)^{-\frac{n}{2\sigma}(\frac{1}{p}- \frac{1}{q})- \frac{a}{2\sigma}-j-1} \big(\|u_0\|_{L^p}+ \|u_1\|_{L^p}\big), $$
for all $a\ge 0$, $j=0,\,1$ and for all space dimensions $n\ge 1$.
\edl
	
The third result contains the global (in time) existence of small data energy solutions to (\ref{equation1.1}).
\bdl \label{theorem3.1}
Let $m \in [1,2)$. We assume the conditions
\begin{align}
&\frac{2}{m} \le p< \ity &\quad \text{ if }\,\,\, &n \le 2\sigma, \label{GN1A1} \\
&\frac{2}{m} \le p\le \frac{n}{n- 2\sigma} &\quad \text{ if }\,\,\, &n \in \Big(2\sigma, \frac{4\sigma}{2-m}\Big]. \label{GN1A2}
\end{align}
Moreover, we suppose the following condition:
\begin{equation} \label{exponent1A}
p> 1+\frac{2m\sigma}{n}.
\end{equation}
Then, there exists a constant $\e>0$ such that for any small data
\begin{equation*}
(u_0,u_1) \in \mathcal{A}^{\sigma}_{m}:= (L^m \cap H^{\sigma}) \times (L^m \cap L^{2}) \text{ satisfying the assumption } \|(u_0,u_1)\|_{\mathcal{A}^{\sigma}_{m}} \le \e,
\end{equation*}
we have a uniquely determined global (in time) small data energy solution
$$ u \in C([0,\ity),H^{\sigma})\cap C^1([0,\ity),L^2) $$
to (\ref{equation1.1}). The following estimates hold:
\begin{align*}
\|u(t,\cdot)\|_{L^2}& \lesssim (1+t)^{-\frac{n}{2\sigma}(\frac{1}{m}-\frac{1}{2})} \|(u_0,u_1)\|_{\mathcal{A}^{\sigma}_{m}}, \\
\big\||D|^{\sigma} u(t,\cdot)\big\|_{L^2}& \lesssim (1+t)^{-\frac{n}{2\sigma}(\frac{1}{m}-\frac{1}{2})- \frac{1}{2}} \|(u_0,u_1)\|_{\mathcal{A}^{\sigma}_{m}}, \\
\|u_t(t,\cdot)\|_{L^2}& \lesssim (1+t)^{-\frac{n}{2\sigma}(\frac{1}{m}-\frac{1}{2})- 1} \|(u_0,u_1)\|_{\mathcal{A}^{\sigma}_{m}}.
\end{align*}
\edl

Next, we obtain the large time behavior of the global solutions to (\ref{equation1.1}).
\bdl \label{theorem3.2}
Under the assumptions of Theorem \ref{theorem3.1} with $m=1$, the global (in time) small data energy solutions to (\ref{equation1.1}) satisfy the following estimate:
\begin{equation}
\big\|\partial_t^j |D|^{k\sigma}\big(u(t,\cdot)- M\,G_\sigma(t,\cdot)\big)\big\|_{L^2}= o\big(t^{-\frac{n}{4\sigma}- \frac{k}{2}-j}\big),
\label{LargetimeEstimate3.2}
\end{equation}
for $j,\,k=0,\,1$ and $(j,k)\neq (1,1)$. Here, we denote the quantity 
\[M:= \int_{\R^n} \big(u_0(y)+u_1(y)\big)dy+ \int_0^\ity \int_{\R^n} |u(\tau,y)|^p dyd\tau.\]
\edl

Finally, we obtain the blow-up result to (\ref{equation1.1}).
\begin{dl} \label{dloptimal4.1}
Let $\sigma \ge 1$ be an integer. We assume the initial data $u_0=0$ and $u_1 \in L^1 \cap L^{2}$ satisfying the following relation:
\begin{equation} \label{optimal1}
\liminf_{R\longrightarrow \ity} \int_{|x| < R} u_1(x) dx >0.
\end{equation}
Moreover, we suppose the condition
\begin{equation} \label{optimal2}
1< p\le 1+\frac{2\sigma}{n}.
\end{equation}
Then, there is no global (in time) energy solution to (\ref{equation1.1}). In other words, we have only local (in time) energy solutions to (\ref{equation1.1}), that is, there exists $T_\e< \ity$ such that
$$ \lim_{t \to T_\e-0} \|(u, u_t)\|_{H^{\sigma} \times L^2}= +\ity. $$
\end{dl}

\begin{nx}
\fontshape{n}
\selectfont
If we choose $m=1$ into Theorem \ref{theorem3.1}, then from Theorem \ref{dloptimal4.1} it is clear to see that the exponent $p$ given by $p=p(n,\sigma)=1+\frac{2\sigma}{n}$ is really critical.
\end{nx}

\par \textbf{The structure of this paper} is organized as follows: Section \ref{Linear estimates} is devoted to estimates for the solutions to (\ref{equation1.2}). In particular, we present some $(L^m \cap L^2)- L^2$ and $L^2- L^2$ estimates for the solutions with $m\in [1,2)$, give the proof of higher order asymptotic expansions of the solutions with weighted initial data, and prove the diffusion phenomenon in the $L^p-L^q$ framework with $1\le p\le q\le \ity$ to (\ref{equation1.2}) in Sections \ref{Sec.LinearEstimates}, \ref{Sec.Expansions} and \ref{Sec.Phenomenon}, respectively. We prove the global (in time) existence of small data energy solutions to (\ref{equation1.1}) in Section \ref{Global existence}, and derive their large time behavior in Section \ref{Large time behavior}. Finally, in Section \ref{Blow-up result}, we show the blow-up result and find the critical exponent as well.

\section{Estimates for the solutions of the linear Cauchy problem} \label{Linear estimates}

\subsection{$L^m \cap L^2-L^2$ and $L^2-L^2$ estimates}\label{Sec.LinearEstimates}
\bmd \label{proposition2.1}
Let $m \in [1,2)$. Then, the Sobolev solutions to (\ref{equation1.2}) satisfy the $(L^m \cap L^2)-L^2$ estimates
\begin{align*}
\big\|\partial_t^j |D|^a u(t,\cdot)\big\|_{L^2} &\lesssim (1+t)^{-\frac{n}{2\sigma}(\frac{1}{m}-\frac{1}{2})- \frac{a}{2\sigma}-j} \big(\|u_0\|_{L^m \cap H^a}+ \|u_1\|_{L^m \cap H^{[a+2(j-1)\sigma]^+}}\big),
\end{align*}
and the $L^2-L^2$ estimates
$$ \big\|\partial_t^j |D|^a u(t,\cdot)\big\|_{L^2} \lesssim (1+t)^{-\frac{a}{2\sigma}-j}\big(\|u_0\|_{H^a}+ \|u_1\|_{H^{[a+2(j-1)\sigma]^+}}\big), $$
for any $a\ge 0$, $j=0,1$ and for all space dimensions $n\ge 1$.
\emd
\begin{proof}
To derive $(L^m \cap L^2)- L^2$ estimates, our strategy is to control $L^2$ norm of the low-frequency part of the solution by $L^m$ norm of the data, whereas its high-frequency part is estimated by using $L^2-L^2$ estimates with a suitable regularity of the data $u_0$ and $u_1$.
\par We shall divide our considerations into two steps. In the first step, let us devote to estimates for low frequencies. We denote $m'$ as a conjugate number of $m$, this is, $\frac{1}{m}+\frac{1}{m'}=1$ and $m_0$ satisfying $\frac{1}{m_0}= \frac{1}{m}- \frac{1}{2}$. By \eqref{equation2.4}, using the formula of Parseval-Plancherel and H\"{o}lder's inequality leads to
\begin{align}
\big\|\partial_t^j |D|^a u_{\text{\fontshape{n}\selectfont low}}(t,\cdot)\big\|_{L^2}&= \big\| |\xi|^a \partial_t^j \widehat{u}(t,\xi) \chi(|\xi|) \big\|_{L^2} \nonumber \\
&\le \Big\|\frac{(-1)^j |\xi|^{a+2j\sigma}e^{-t|\xi|^{2\sigma}}}{1-|\xi|^{2\sigma}}\big(\widehat{u_0}(\xi)+ \widehat{u_1}(\xi)\big) \chi(|\xi|)\Big\|_{L^2} \nonumber \\
&\qquad + \Big\|\frac{(-1)^{j+1}|\xi|^a e^{-t}}{1-|\xi|^{2\sigma}}\big(|\xi|^{2\sigma} \widehat{u_0}(\xi)+ \widehat{u_1}(\xi)\big)\chi(|\xi|)\Big\|_{L^2} \nonumber \\
&\lesssim \Big\|\frac{|\xi|^{a+2j\sigma}e^{-t|\xi|^{2\sigma}}\chi(|\xi|)}{1-|\xi|^{2\sigma}}\Big\|_{L^{m_0}} \|\widehat{u_0}+ \widehat{u_1}\|_{L^{m'}} \label{pro2.1.1}\\
&\qquad + e^{-t}\,\Big\|\frac{|\xi|^{a+2\sigma} \chi(|\xi|)}{1-|\xi|^{2\sigma}}\Big\|_{L^{m_0}} \|\widehat{u_0}\|_{L^{m'}}+ e^{-t}\,\Big\|\frac{|\xi|^a \chi(|\xi|)}{1-|\xi|^{2\sigma}}\Big\|_{L^{m_0}} \|\widehat{u_1}\|_{L^{m'}}. \label{pro2.1.2}
\end{align}
For the sake of Young-Hausdorff inequality, we can control $ \|\widehat{u_0}+ \widehat{u_1}\|_{L^{m'}}$, $\|\widehat{u_0}\|_{L^{m'}}$ and $\|\widehat{u_1}\|_{L^{m'}}$ by $\|u_0+ u_1\|_{L^m}$, $\|u_0\|_{L^m}$ and $\|u_1\|_{L^m}$, respectively. Hence, we only have to estimate $L^{m_0}$ norm of the multipliers. It is clear to see that the last two $L^{m_0}$ norms are bounded. Taking account of the first $L^{m_0}$ norm, we apply Lemma \ref{LemmaL1normEstimate} to obtain immediately the following estimate:
\begin{equation}
\Big\|\frac{|\xi|^{a+2j\sigma}e^{-t|\xi|^{2\sigma}}\chi(|\xi|)}{1-|\xi|^{2\sigma}}\Big\|_{L^{m_0}} \lesssim  \big\||\xi|^{a+2j\sigma}e^{-t|\xi|^{2\sigma}}\chi(|\xi|)\big\|_{L^{m_0}} \lesssim  (1+t)^{-\frac{n}{2\sigma}(\frac{1}{m}-\frac{1}{2})- \frac{a}{2\sigma}-j}.
\label{pro2.1.3}
\end{equation}
Therefore, from \eqref{pro2.1.1} to \eqref{pro2.1.3} we arrive at
\begin{equation}
\big\|\partial_t^j |D|^a u_{\text{\fontshape{n}\selectfont low}}(t,\cdot)\big\|_{L^2}\lesssim (1+t)^{-\frac{n}{2\sigma}(\frac{1}{m}-\frac{1}{2})- \frac{a}{2\sigma}-j}\|u_0+ u_1\|_{L^m}+ e^{-t}\,\big(\|u_0\|_{L^m}+ \|u_1\|_{L^m}\big).
\label{pro2.1.4}
\end{equation}
Next, let us turn to estimate the solution and some of its derivatives to (\ref{equation1.2}) for large frequencies. Thanks to \eqref{equation2.6}, we apply again the formula of Parseval-Plancherel and use a suitable regularity of the data $u_0$ and $u_1$ to find the following estimate:
\begin{equation}
\big\|\partial_t^j |D|^a u_{\text{\fontshape{n}\selectfont high}}(t,\cdot)\big\|_{L^2} \lesssim e^{-t}\big(\|u_0\|_{H^a}+ \|u_1\|_{H^{[a- 2\sigma]^+}}\big)+ e^{-t}\|u_0+ u_1\|_{H^{[a+2(j-1)\sigma]^+}}.
\label{pro2.1.5}
\end{equation}
From \eqref{pro2.1.4} and \eqref{pro2.1.5} we may conclude all the desired estimates. Summarizing, the proof of Proposition \ref{proposition2.1} is completed.
\end{proof}

\begin{nx}
\fontshape{n}
\selectfont
Here we want to underline that the exponential decay $e^{-t}$  appearing in the proof of Proposition \ref{proposition2.1} is better than the potential decay. Since we have in mind that the characteristic roots $\lambda_{\pm}$ are negative in the middle zone $|\xi| \in \big\{\e,\, \frac{1}{\e}\big\}$ with a sufficiently small positive $\e$, the corresponding estimates yield an exponential decay in this zone, too.
\end{nx}

\subsection{Asymptotic profile and higher order asymptotic expansions} \label{Sec.Expansions}
In this subsection we obtain higher order asymptotic expansions of the solution to \eqref{equation1.2} in the Fourier space (see Theorem~\ref{TheoremLinear}). 
It is necessary to analyze the solution formulas \eqref{equation2.4}-\eqref{sol} in the low-frequency region. 
In order to state Lemma~\ref{KeyLemma}, which is a key to derive Theorem~\ref{TheoremLinear}, we prepare the following notation and definition. 

For $f\in L^{1,\gamma}(\R^n)$, 
we put
\[
M_\alpha (f)
:=\frac{(-1)^{|\alpha|}}{\alpha!}
\int_{\R^n} x^\alpha f(x)\,dx, 
\qquad
|\alpha|\le[\gamma].
\]

\bdn
For $0<\sigma\in\R$, set $\mathfrak{S}:=\{2\sigma\ell+j : \ell, j\in\N_0\}$. 
We define the following function inductively: 
\begin{align*}
\lambda(k)
:=
\begin{cases}
\min \mathfrak{S}=0 &\text{ for }\,\,\, k=0, \\
\min \mathfrak{S}\setminus\{\lambda(j):0\le j\le k-1\} &\text{ for }\,\,\, k\in\N.
\end{cases}
\end{align*}
\edn

\bnx
\fontshape{n}
\selectfont
\begin{enumerate}[leftmargin=*]
\item 
Now we deal with the case of $\sigma\ge1$ and thus it holds that 
\[
\lambda(k)=k, 
\qquad
k=0,1,2.
\]
\item Since $\N_0\subset \mathfrak{S}$, 
it holds that $0<\lambda(k+1)-\lambda(k)\le1$ for all $k\in\N_0$. 
\end{enumerate}
\label{Remark2.2}
\enx

\bdn
Let $\sigma\ge1$ and $k\in\N_0$. 
We define 
\begin{align}
A^{\sigma,v_0}_k(\xi)
:=\sum_{0\le 2\sigma\ell+j\le\lambda(k)}
\left(
|\xi|^{2\sigma\ell}
\sum_{|\alpha|=j}M_\alpha(v_0)(i\xi)^\alpha
\right).
\label{XXX}
\end{align}
The sum $\displaystyle{\sum_{0\le 2\sigma\ell+j\le\lambda(k)}}$ is taken over all $\ell, j\in\N_0$ satisfying $0\le2\sigma\ell+j\le \lambda(k)$.
\label{Def2.2}
\edn

\bnx
\fontshape{n}
\selectfont
\begin{enumerate}[leftmargin=*]
\item The function~\eqref{XXX} itself can be defined for $v_0 \in L^{1,\gamma}(\R^n)$ with $\gamma\ge[\lambda(k)]$. 
For later necessarity, it suffices to consider \eqref{XXX} for $v_0 \in L^{1,\gamma}(\R^n)$ with $\gamma\ge\lambda(k)$. 
\item Recall Remark~\ref{Remark2.2} to confirm 
\[
A_0^{\sigma, v_0}(\xi)
=M_0(v_0), 
\qquad
A_1^{\sigma, v_0}(\xi)
=\sum_{|\alpha|\le 1}
M_\alpha(v_0)(i\xi)^\alpha, 
\]
\begin{align*}
A_2^{\sigma, v_0}(\xi)=
\begin{cases}
\displaystyle{
\sum_{|\alpha|\le 2}
M_\alpha(v_0)(i\xi)^\alpha
} 
&\text{ if }\,\,\, \sigma>1,  \\[22pt]
|\xi|^2 M_0(v_0)
+\displaystyle{
\sum_{|\alpha|\le2}
M_\alpha(v_0)(i\xi)^\alpha
} 
&\text{ if }\,\,\, \sigma=1. 
\end{cases}
\end{align*}
We can easily see that the difference between the heat flow and equation~\eqref{equation1.2} will come out first in the $k$-th order expansion $A_k^{\sigma,v_0}$ with $k\le\sigma<k+1$. 
However, it seems difficult to write down $A_k^{\sigma, v_0}$ for large $k\in\N_0$. 
In \cite{IkehataMichihisa}, the case of $\sigma=1$ was completely investigated. 
\end{enumerate}
\enx

\bbd \label{KeyLemma}
Let $n\ge1$ and $v_0 \in L^{1,\gamma}(\R^n)$ with $\gamma\ge0$. 
For this $\gamma$, there exists a unique number $k\in\N_0$ satisfying $\lambda(k)\le\gamma<\lambda(k+1)$. 
Then, it holds that 
\begin{align}
\big|F^{\sigma,v_0}(\xi)-A^{\sigma,v_0}_k(\xi)\big|
&\lesssim |\xi|^\gamma \|v_0\|_{L^{1,\gamma}}
\label{YYY}
\end{align}
for $\xi\in\R^n$ with $|\xi|\le1/2$. 
\ebd
\begin{proof}
For given $\gamma\ge0$, we can find $k\in\N_0$ satisfying $\lambda(k)\le\gamma<\lambda(k+1)$. 
In this setting we have 
\[
[\gamma]=[\lambda(k)].
\]
If not, then there exists an integer $b\in\N$ such that $\lambda(k)<b\le\gamma<\lambda(k+1)$. 
All natural numbers are included in $\{\lambda(j) : j\in\N_0\}$ and so this is a contradiction. 

It follows that 
\begin{align*}
F^{\sigma,v_0}(\xi)& =\left(
\sum_{\ell=0}^{[\lambda(k)]} |\xi|^{2\sigma\ell} 
+\frac{|\xi|^{2\sigma([\lambda(k)]+1)}}{1-|\xi|^{2\sigma}}
\right)
\left\{
\sum_{j=0}^{[\lambda(k)]}
\sum_{|\alpha|=j}
M_\alpha(v_0)(i\xi)^\alpha
+
\left(
\widehat{v_0}-
\sum_{|\alpha|\le[\gamma]}
M_\alpha(v_0)(i\xi)^\alpha
\right)
\right\} \\
& =\left(
\sum_{\ell=0}^{[\lambda(k)]} |\xi|^{2\sigma\ell} 
\right)
\left(
\sum_{j=0}^{[\lambda(k)]}
\sum_{|\alpha|=j}
M_\alpha(v_0)(i\xi)^\alpha
\right) \\
& \qquad 
+\left(
\sum_{\ell=0}^{[\lambda(k)]} |\xi|^{2\sigma\ell} 
\right)
\left(
\widehat{v_0}-
\sum_{|\alpha|\le[\gamma]}
M_\alpha(v_0)(i\xi)^\alpha
\right) 
+\frac{|\xi|^{2\sigma([\gamma]+1)}}{1-|\xi|^{2\sigma}}
\widehat{v_0}. 
\end{align*}
From \eqref{moment1}, we see that 
\[
\left|
\left(
\sum_{\ell=0}^{[\lambda(k)]} |\xi|^{2\sigma\ell} 
\right)
\left(
\widehat{v_0}-
\sum_{|\alpha|\le[\gamma]}
M_\alpha(v_0)(i\xi)^\alpha
\right)
\right|
\lesssim 
|\xi|^\gamma 
\|v_0\|_{L^{1,\gamma}}
\]
for $\xi\in\R^n$ with $|\xi|\le 1/2$. 
Since 
\[
2\sigma([\gamma]+1)>\gamma,
\]
one easily sees that 
\[
\left|
\frac{|\xi|^{2\sigma([\gamma]+1)}}{1-|\xi|^{2\sigma}}
\widehat{v_0}
\right|
\lesssim |\xi|^\gamma \|v_0\|_{L^1}
\]
for $\xi\in\R^n$ with $|\xi|\le 1/2$. 
Thus, we arrive at 
\begin{align*}
\left|
F^{\sigma,v_0}(\xi)
-\left(
\sum_{\ell=0}^{[\lambda(k)]} |\xi|^{2\sigma\ell} 
\right)
\left(
\sum_{j=0}^{[\lambda(k)]}
\sum_{|\alpha|=j}
M_\alpha(v_0)(i\xi)^\alpha
\right)
\right|
& \lesssim |\xi|^\gamma \|v_0\|_{L^{1,\gamma}}
+|\xi|^{2\sigma([\gamma]+1)} \|v_0\|_{L^1} \\
& \lesssim |\xi|^\gamma \|v_0\|_{L^{1,\gamma}}
\end{align*}
for $\xi\in\R^n$ with $|\xi|\le 1/2$. 
If $[\lambda(k)]=0$, i.e., $k=0$, then 
\[
\left(
\sum_{\ell=0}^{[\lambda(k)]} |\xi|^{2\sigma\ell} 
\right)
\left(
\sum_{j=0}^{[\lambda(k)]}
\sum_{|\alpha|=j}
M_\alpha(v_0)(i\xi)^\alpha
\right)
=M_0(v_0)
=A_0^{\sigma, v_0}(\xi). 
\]
So in this case we obtain the lemma. 
On the other hand, if $[\lambda(k)]\ge1$, we have
\[
\lambda(k)<(2\sigma+1)[\lambda(k)]
\in \mathfrak{S}.
\]
Hence, it follows that 
\begin{align*}
& \left(
\sum_{\ell=0}^{[\lambda(k)]} |\xi|^{2\sigma\ell} 
\right)
\left(
\sum_{j=0}^{[\lambda(k)]}
\sum_{|\alpha|=j}
M_\alpha(v_0)(i\xi)^\alpha
\right) \\
&\qquad =\sum_{\substack{0\le2\sigma\ell+j\le\lambda(k), \\
0\le j\le[\lambda(k)]}}
\left(
|\xi|^{2\sigma\ell} 
\sum_{|\alpha|=j}
M_\alpha(v_0)(i\xi)^\alpha
\right) 
+\sum_{\substack{\lambda(k)<2\sigma\ell+j\le(2\sigma+1)[\lambda(k)], \\
0\le j\le[\lambda(k)]}}
\left(
|\xi|^{2\sigma\ell} 
\sum_{|\alpha|=j}
M_\alpha(v_0)(i\xi)^\alpha
\right) \\
&\qquad =A_k^{\sigma, v_0}(\xi)
+\sum_{\substack{\lambda(k+1)\le2\sigma\ell+j\le(2\sigma+1)[\lambda(k)], \\
0\le j\le[\gamma]}}
\left(
|\xi|^{2\sigma\ell} 
\sum_{|\alpha|=j}
M_\alpha(v_0)(i\xi)^\alpha
\right).
\end{align*} 
Thus, it holds 
\[
\left|
\sum_{\substack{\lambda(k+1)\le2\sigma\ell+j\le(2\sigma+1)[\lambda(k)], \\
0\le j\le[\gamma]}}
\left(
|\xi|^{2\sigma\ell} 
\sum_{|\alpha|=j}
M_\alpha(v_0)(i\xi)^\alpha
\right)
\right|
\lesssim |\xi|^{\lambda(k+1)}\|v_0\|_{L^{1,[\gamma]}}
\lesssim |\xi|^\gamma \|v_0\|_{L^{1,[\gamma]}}
\]
for $\xi\in\R^n$ with $|\xi|\le 1/2$.  
Therefore, we obtain \eqref{YYY}. 
\end{proof}

\begin{proof}[Proof of Theorem \ref{TheoremLinear}]
It follows from Lemma~\ref{KeyLemma} with Lemma~\ref{LemmaL1normEstimate} that 
\begin{align*}
\big\|F^{\sigma,v_0}(\xi)e^{-t|\xi|^{2\sigma}}
-A_k^{\sigma,v_0}(\xi)e^{-t|\xi|^{2\sigma}}\big\|_{L^2(|\xi|\le1/2)}
\lesssim (1+t)^{-\frac{n}{4\sigma}-\frac{\gamma}{2\sigma}}
\|v_0\|_{L^{1,\gamma}}, 
\qquad
t\ge0, 
\end{align*}
which implies 
\begin{align}
\big\|\hat{u}(t,\xi)
-A_k^{\sigma,v_0}(\xi)e^{-t|\xi|^{2\sigma}}\big\|_{L^2(|\xi|\le1/2)}
\lesssim (1+t)^{-\frac{n}{4\sigma}-\frac{\gamma}{2\sigma}}
(\|v_0\|_{L^{1,\gamma}}
+\|u_0\|_{L^1}+\|u_1\|_{L^1}), 
\qquad
t\ge0, 
\label{LOW}
\end{align}
We employ \eqref{sol} to derive
\[
|\hat{u}(t,\xi)|
\lesssim e^{-t} |\widehat{u_0}(\xi)|
+te^{-t} (|\widehat{u_0}(\xi)|+|\widehat{u_1}(\xi)|)
\]
for $t\ge0$ and $\xi\in\R^n$ with $|\xi|\ge 1/2$. 
We can easily see that 
\[
\big|
A_k^{\sigma,v_0}(\xi)e^{-t|\xi|^{2\sigma}}
\big|
\lesssim |\xi|^{\lambda(k)}
\exp\left(
-\frac{|\xi|^{2\sigma}}{2}
\right)
\exp\left(
-\frac{t}{2^{2\sigma+1}}
\right)
\|v_0\|_{L^{1,[\gamma]}}
\]
for $t\ge1$ and $\xi\in\R^n$ with $|\xi|\ge 1/2$. 
Thus, there exists a constant $c\in(0,2^{-(2\sigma+1)}]$ such that 
\begin{align}
\big\|\hat{u}(t,\xi)-A_k^{\sigma,v_0}(\xi)e^{-t|\xi|^{2\sigma}}\big\|_{L^2(|\xi|\ge1/2)}
\lesssim e^{-ct}
\big(\|u_0\|_{L^2}+\|u_1\|_{L^2}+\|v_0\|_{L^{1,[\gamma]}}\big),
\qquad
t\ge1.
\label{HIGH}
\end{align}
Inequalities~\eqref{LOW} and \eqref{HIGH} give \eqref{AE1}. 

By \eqref{moment2}, we can also obtain \eqref{AE2}. 
See the corresponding proof in \cite{IkehataMichihisa} and details are left to the reader.
\end{proof}

Recalling \eqref{AE1} in Theorem~\ref{TheoremLinear} with $\gamma=1$ we can obtain the following corollary since 
\[
\left\|
\sum_{|\alpha|=1}M_\alpha(v_0)(i\xi)^\alpha
e^{-t|\xi|^{2\sigma}}
\right\|_{L^2}
\lesssim t^{-\frac{n}{4\sigma}- \frac{1}{2\sigma}}\|v_0\|_{L^{1,1}},
\qquad
t>0.
\]
\bhq \label{corollary2.2.1}
Let $n\ge1$ and $u$ be the function defined by \eqref{equation2.4}-\eqref{sol}. If $u_0,\, u_1\in L^{1,1}(\R^n)\cap L^2(\R^n)$, it holds 
$$ \big\|u(t,\cdot)- (P_0+P_1)G_\sigma(t,\cdot)\big\|_{L^2} 
\lesssim t^{-\frac{n}{4\sigma}- \frac{1}{2\sigma}} \big(\|u_0\|_{L^{1,1} \cap L^2}+ \|u_1\|_{L^{1,1} \cap L^2}\big), 
\qquad t\ge 1. $$ 
\ehq

We may conclude the following optimal result at the end of this subsection.
\bhq\label{corollary2.1}
Let $n\ge1$ and $u$ be the function defined by \eqref{equation2.4}-\eqref{sol}. If $u_0,\, u_1\in L^{1}(\R^n)\cap L^2(\R^n)$, it holds 
$$ C_1(|P_0+ P_1|)t^{-\frac{n}{4\sigma}} \le \|u(t,\cdot)\|_{L^2} 
\le C_2 t^{-\frac{n}{4\sigma}} \big(\|u_0\|_{L^{1} \cap L^2}+ \|u_1\|_{L^{1} \cap L^2}\big), 
\qquad 
t\ge 1. $$
Here, $C_1>0$ and $C_2 >0$ are constants independent of $t$ and the initial data.
\ehq
\begin{proof}
The second inequality can be easily given with the aid of \eqref{equation2.4}-\eqref{sol} or \eqref{AE1} with $\gamma=0$. 
So we confirm the first inequality. 
To do so, it suffices to consider the case that $P_0+P_1\not=0$. 
In this situation one has  
\begin{align*}
\|u(t,\cdot)\|_2
& \ge \|u(t,\xi)\|_{L^2(|\xi|\le1/2)} \\
& \ge |P_0+P_1| \|e^{-t|\xi|^{2\sigma}}\|_{L^2(|\xi|\le1/2)}
-o(t^{-\frac{n}{4\sigma}})
\end{align*}
as $t\to\infty$. 
Here, we used \eqref{AE2} with $\gamma=0$, that is, 
\[
\lim_{t\to\infty}
t^{\frac{n}{4\sigma}}
\big\|\hat{u}(t,\xi)- (P_0+P_1)e^{-t|\xi|^{2\sigma}}\big\|_{L^2(|\xi|\le1/2)}
=0.
\]
For $t\ge1$, we have 
\[
\|e^{-t|\xi|^{2\sigma}}\|_{L^2(|\xi|\le1/2)}
=t^{-\frac{n}{4\sigma}}\left(
\int_{|\eta|\le t^\frac{1}{2\sigma}/2} 
e^{-2|\eta|^{2\sigma}}
\,d\eta
\right)^{\frac{1}{2}}
\ge \left(
\int_{|\eta|\le 1/2} 
e^{-2|\eta|^{2\sigma}}
\,d\eta
\right)^{\frac{1}{2}}
t^{-\frac{n}{4\sigma}}
\]
and thus the corollary is obtained. 
\end{proof}

\subsection{Diffusion phenomenon in the $L^p-L^q$ framework} \label{Sec.Phenomenon}
In this section, we shall discuss a relation between the solutions to \eqref{equation1.2} and to the anomalous diffusion equation \eqref{equation1.5}. Clearly, if we consider the power $\sigma= 1$, then it corresponds to the classical heat equation. Let $(u_0,\,u_1)$ be given data to \eqref{equation1.2}. If we can find an appropriate data $v_0$ to \eqref{equation1.5} such that the difference of the corresponding solutions $u(t,\cdot)- v(t,\cdot)$ possesses a decay rate as $t \to \ity$ in a suitable norm, then one says that the asymptotic behavior of both the solutions is the same for large time. This effect is the so-called diffusion phenomenon.
\par The application of partial Fourier transform to (\ref{equation1.5}) leads to
\begin{equation}
\begin{cases}
\hat{v}_t+ |\xi|^{2\sigma} \hat{v}= 0, \\
\hat{v}(0,\xi)= \widehat{v_0}(\xi).
\end{cases}
\label{equation2.3.2}
\end{equation}
Hence, the solution to (\ref{equation2.3.2}) is written by the formula
\begin{equation}
\hat{v}(t,\xi)= e^{-t|\xi|^{2\sigma}}\widehat{v_0}(\xi).
\label{equation2.3.3}
\end{equation}
Recalling the abbreviation $G_\sigma(t,x)$ gives
\begin{equation}
v(t,x)= G_\sigma(t,x) \ast_x v_0(x).
\label{equation2.3.4}
\end{equation}
Because of the presence of the frictional damping term in \eqref{equation1.2}, considering large frequencies $|\xi|$ and large time $t$ we can conclude some exponential decay estimates for the solutions to \eqref{equation1.2} as we derived \eqref{pro2.1.5}. Moreover, it holds that we may arrive at an exponential decay for the solutions to \eqref{equation1.5} for large frequencies $|\xi|$ and for large times $t$. This means that the difference between the solutions to \eqref{equation1.2} and \eqref{equation1.5} decays. Hence, the obtained decay rate of the difference is optimal. For this reason, it is suitable to focus our attentions on estimates for the difference localized to small frequencies.
\par Our approach to prove Theorem \ref{theorem2.1} is based on applying the following two auxiliary results. The first one is a result for radial convolution kernels.
\bbd[Lemma 3.1 in \cite{DabbiccoEbert2014}] \label{lemma2.3}
Let $K(t,x)$ be a radial convolution kernel of the form
$$ K(t,x) \ast_{(x)} h(x):= F^{-1}\big(f(|\xi|)\,e^{-g(|\xi|)t}\,\hat{h}(\xi)\big), $$
with compactly supported $h$, where $f$ and $g$ satisfy the following conditions:
\begin{align*}
\big|f^{(k)}(\rho)\big|&\lesssim \rho^{\alpha-k}, \\
\big|g^{(k)}(\rho)\big|&\lesssim \rho^{-k}g(\rho), \\
g(\rho)&\approx \rho^{\beta},  
\end{align*}
for some $\alpha> -1$, $\beta>0$ and $k\le [(n+3)/2]$. Then, it holds
$$ \|K_{\text{\fontshape{n}\selectfont low}}(t,\cdot) \ast_{(x)} h\|_{L^q} \lesssim (1+t)^{-\frac{n}{\beta}(\frac{1}{p}- \frac{1}{q})- \frac{\alpha}{\beta}}\|h\|_{L^p}, $$
provided that for any
$$\begin{cases}
1\le p\le q\le \ity &\text{ if }\,\,\, \alpha>0 \text{ or } f \text{ is a nonzero constant}, \\
1\le p< q\le \ity &\text{ if }\,\,\, \alpha=0 \text{ and } f \text{ is not constant}, \\
1\le p\le q\le \ity &\text{ if }\,\,\, -1<\alpha<0 \text{ such that } \frac{1}{p}- \frac{1}{q}\ge \frac{-\alpha}{n}.
\end{cases}$$
\ebd
The second result is related to $L^r$ estimates for multipliers.
\bbd \label{lemma2.4}
Let $n\ge 1$ and $r\in [1,\ity]$. Then, for all $a>0$ it holds
\begin{equation}
K:= \mathcal{F}^{-1}\Big(\frac{|\xi|^a}{1-|\xi|^{2\sigma}}\chi(|\xi|)\Big) \in L^r.
\label{lem2.4.1}
\end{equation}
\ebd
\begin{proof}
First it is clear to see that $K \in L^{\ity}$. For this reason, in order to prove $K \in L^r$ for all $r\in [1,\ity]$, we only indicate $K \in L^1$ and apply an interpolation argument. Indeed, we shall split our considerations into two cases. In the first case of $|x|\le 1$, it is obvious to conclude the desired statement. Let us devote to the second case of $|x|\ge 1$. Because the function in the parenthesis in (\ref{lem2.4.1}) is radially symmetric with respect to $\xi$, the inverse Fourier transform is radially symmetric with respect to $x$, too. Applying the modified Bessel functions from Proposition \ref{PropertiesModifiedBesselfunctions} we obtain
\begin{equation}
K(x)= c\int_0^\ity \frac{r^a}{1-r^{2\sigma}}\chi(r) r^{n-1} \tilde{J}_{\frac{n}{2}-1}(r|x|) dr. \label{lem2.4.2}
\end{equation}
Let us consider odd spatial dimensions $n=2m+1,\, m \ge 1$. Introducing the vector field $Xf(r):= \frac{d}{dr}\big(\frac{1}{r}f(r)\big)$ and carrying out $m+1$ steps of partial integration we derive
\begin{equation}
K(x)= -\frac{c}{|x|^{n}}\int_0^\ity \partial_r \Big(X^m \Big( \frac{1}{1-r^{2\sigma}} \chi(r) r^{a+2m}\Big)\Big) \sin(r|x|)dr. \label{lem2.4.3}
\end{equation}
A straightforward computation gives
\begin{align*}
K(x)&= \sum_{j=0}^m \sum_{k=0}^{j+1}\frac{c_{jk}}{|x|^{n}}\int_0^\ity \partial_r^{j+1-k} \Big(\frac{1}{1-r^{2\sigma}}\Big)\, \chi^{(k)}(r) r^{a+j} \sin(r|x|)dr\\
&\quad + \sum_{j=0}^m \sum_{k=0}^j \frac{c_{jk}}{|x|^{n}}\int_0^\ity \partial_r^{j-k} \Big(\frac{1}{1-r^{2\sigma}}\Big)\, \chi^{(k+1)}(r) r^{a+j} \sin(r|x|)dr\\
&\quad + \sum_{j=1}^m \sum_{k=0}^j \frac{c_{jk}}{|x|^{n}}\int_0^\ity \partial_r^{j-k} \Big(\frac{1}{1-r^{2\sigma}}\Big)\, \chi^{(k)}(r) r^{a+j-1} \sin(r|x|)dr
\end{align*}
with some constants $c_{jk}$. Hence, it is reasonable to estimate the integrals
\begin{equation}
K_{j,k}(x):= \int_0^\ity \partial_r^{j+1-k} \Big(\frac{1}{1-r^{2\sigma}}\Big)\, \chi^{(k)}(r) r^{a+j} \sin(r|x|)dr. \label{lem2.4.4}
\end{equation}
For $k\ge 1$, due to the fact that for all $l\ge 0$
$$ \Big|\partial_r^{l} \Big(\frac{1}{1-r^{2\sigma}}\Big)\Big| \le C_l $$
on the support of the derivatives of $\chi$, we perform one more step of partial integration to get
\begin{equation}
K_{j,k}(x)|\lesssim |x|^{-(n+1)}.
\label{lem2.4.5}
\end{equation}
For $k=0$, because of the small values of $r$, we arrive at
$$ \Big|\partial_r^{l} \Big(\frac{1}{1-r^{2\sigma}}\Big)\Big| \lesssim
\begin{cases}
1 &\text{ if }\,\,\, l=0, \\
r^{2\sigma-l} &\text{ if }\,\,\, l\ge 1,
\end{cases}  $$
on the support of $\chi(r)$. Consequently, it deduces for small $r$ and $j= 0,\cdots,m$ the estimates
$$\Big|\partial_r^{j+1} \Big(\frac{1}{1-r^{2\sigma}}\Big)\, \chi(r) r^{a+j}\Big| \lesssim r^{a+2\sigma-1} $$
on the support of $\chi(r)$. By dividing the integral (\ref{lem2.4.4})  into two parts, on the one hand, we have
\begin{equation}
\Big |\int_0^{\frac{\pi}{2|x|}} \partial_r^{j+1} \Big(\frac{1}{1-r^{2\sigma}}\Big)\, \chi(r) r^{a+j} \sin(r|x|)dr \Big| \lesssim \frac{1}{|x|^{a+2\sigma}}. \label{lem2.4.6}
\end{equation}
On the other hand, after carrying out one more step of partial integration in the remaining integral we can proceed as follows:
\begin{align}
&\Big|\int_{\frac{\pi}{2|x|}}^\ity \partial_r^{j+1} \Big(\frac{1}{1-r^{2\sigma}}\Big)\, \chi(r) r^{a+j} \sin(r|x|)dr \Big| \nonumber \\
&\qquad \lesssim \frac{1}{|x|}\Big|\partial_r^{j+1} \Big(\frac{1}{1-r^{2\sigma}}\Big)\, \chi(r) r^{a+j} \cos(r|x|) \Big|_{r=\frac{\pi}{2|x|}}^\ity \nonumber \\
&\qquad \quad+ \frac{1}{|x|}\int_{\frac{\pi}{2|x|}}^\ity \Big|\partial_r \Big( \partial_r^{j+1} \Big(\frac{1}{1-r^{2\sigma}}\Big)\, \chi(r) r^{a+j}\Big) \cos(r|x|)\Big| \,dr \lesssim \frac{1}{|x|}\int_{\frac{\pi}{2|x|}}^1 r^{a+2\sigma-2}\,dr \lesssim \frac{1}{|x|}. \label{lem2.4.7}
\end{align}
Here we notice that for all $j= 0,\cdots,m$ and for small $|\xi|$ it holds
$$\Big|\partial_r \Big( \partial_r^{j+1} \Big(\frac{1}{1-r^{2\sigma}}\Big)\, \chi(r) r^{a+j}\Big)\Big| \lesssim r^{a+2\sigma-2}. $$
From (\ref{lem2.4.3}) to (\ref{lem2.4.7}) we have produced term $|x|^{-(n+1)}$ which guarantees the $L^1$ property in $x$. Therefore, we may conclude $K \in L^1$ for all $n=2m+1$. \medskip

\noindent Let us consider even spatial dimensions $n=2m,\, m \ge 1$. After carrying out $m-1$ steps of partial integration we derive
\begin{align}
K(x)&= \frac{c}{|x|^{2m-2}}\int_0^\ity X^{m-1}\Big( \frac{1}{1-r^{2\sigma}} \chi(r) r^{a+2m-1} \Big) \tilde{J}_0(r|x|) dr  \nonumber \\
&= \sum_{j=0}^{m-1}\frac{c_j}{|x|^{2m-2}}\int_0^\ity \partial_r^j \Big( \frac{1}{1-r^{2\sigma}} \chi(r) r^a \Big) r^{j+1} \tilde{J}_0(r|x|) dr =:\sum_{j=0}^{m-1} c_j K_j(x). \label{lem2.4.8}
\end{align}
Using the first rule of the modified Bessel functions for $\mu=1$ and the fifth rule for $\mu=0$ from Proposition \ref{PropertiesModifiedBesselfunctions}, after two more steps of partial integration we arrive at
\begin{equation}
K_0(x)= -\frac{1}{|x|^{n}}\int_0^1 \partial_r \Big( \partial_r \Big( \frac{1}{1-r^{2\sigma}} \chi(r) r^a \Big) r \Big) \tilde{J}_0(r|x|) dr. \label{lem2.4.9}
\end{equation}
Due to small $r$, it implies the following inequality:
$$ \Big|\partial_r \Big( \partial_r \Big( \frac{1}{1-r^{2\sigma}} \chi(r) r^a \Big) r \Big)\Big| \lesssim r^{a-1} $$
on the support of $\chi(r)$. By the aid of the estimates $|\tilde{J}_0(s)| \le C$ for $s \in [0,1]$, and $|\tilde{J}_0(s)| \le Cs^{-\frac{1}{2}}$ for $s>1$, we get
\begin{equation}
\Big| \int_0^{\frac{1}{|x|}} \partial_r \Big( \partial_r \Big( \frac{1}{1-r^{2\sigma}} \chi(r) r^a \Big) r \Big) \tilde{J}_0(r|x|) dr \Big| \lesssim \int_0^{\frac{1}{|x|}} r^{a-1} dr \lesssim \frac{1}{|x|^a},
\label{lem2.4.10}
\end{equation}
and
\begin{align}
&\Big| \int_{\frac{1}{|x|}}^1 \partial_r \Big( \partial_r \Big( \frac{1}{1-r^{2\sigma}} \chi(r) r^a \Big) r \Big) \tilde{J}_0(r|x|) dr \Big| \nonumber \\ 
&\qquad \lesssim \frac{1}{|x|^{\frac{1}{2}}}\int_{\frac{1}{|x|}}^1 r^{a-\frac{3}{2}} dr \lesssim
\begin{cases}
\frac{1}{|x|^{\frac{1}{2}}}\Big(1+ \frac{1}{|x|^{a- \frac{1}{2}}}\Big) &\text{ if }\,\,\, a \neq \frac{1}{2} \\
\frac{1}{|x|^{\frac{1}{2}}}\log(|x|) &\text{ if }\,\,\, a= \frac{1}{2}
\end{cases}
\,\,\lesssim \frac{1}{|x|^{\e}}, \label{lem2.4.11}
\end{align}
with a sufficiently small positive constant $\e$, respectively. As a result, from (\ref{lem2.4.9}) to (\ref{lem2.4.11}) we obtain $K_0 \in L^1$. Let $j \in [1,m-1]$ be an integer. By applying again the first rule of the modified Bessel functions for $\mu=1$ and the fifth rule for $\mu=0$ from Proposition \ref{PropertiesModifiedBesselfunctions} and carrying out partial integration we can re-write $K_j(x)$ in (\ref{lem2.4.8}) as follows:
\begin{align*}
K_j(x)= &-\frac{1}{|x|^{2m}}\int_0^\ity \partial_r \Big( \partial_r^{j+1}\Big( \frac{1}{1-r^{2\sigma}} \chi(r) r^a \Big) r^{j+1}\Big) \tilde{J}_0(r|x|) dr\\
&-\frac{j}{|x|^{2m}}\int_0^\ity \partial_r \Big( \partial_r^j \Big( \frac{1}{1-r^{2\sigma}} \chi(r) r^a \Big) r^j \Big) \tilde{J}_0(r|x|) dr.
\end{align*}
Repeating an analogous treatment as we did for $K_0=K_0(x)$ we derive $K_j \in L^1$ for $j=1,\cdots,m-1$. Therefore, we may conclude the desired estimate $K \in L^1$ for all $n=2m$. Summarizing, this completes the proof of Lemma \ref{lemma2.4}.
\end{proof}

\begin{proof}[Proof of Theorem \ref{theorem2.1}]
Thanks to the solution formulas \eqref{equation2.4} and \eqref{equation2.3.3}, we obtain
\small
\begin{align}
\big\|\partial_t^j |D|^a \big(u_{\text{\fontshape{n}\selectfont low}}(t,\cdot)- v_{\text{\fontshape{n}\selectfont low}}(t,\cdot)\big)\big\|_{L^q}&= \big\|\mathcal{F}^{-1}\big(|\xi|^a \partial_t^j \big(\hat{u}(t,\xi)- \hat{v}(t,\xi)\big)\chi(|\xi|)\big)\big\|_{L^q} \nonumber \\ 
&\lesssim e^{-t}\,\Big\|\mathcal{F}^{-1}\Big(\frac{|\xi|^{a+2\sigma}}{1- |\xi|^{2\sigma}}\chi(|\xi|)\widehat{u_0}(|\xi|)+ \frac{|\xi|^a}{1- |\xi|^{2\sigma}}\chi(|\xi|)\widehat{u_1}(|\xi|)\Big)\Big\|_{L^q} \label{pro2.4.1} \\
&\quad+ \Big\|\mathcal{F}^{-1}\Big(e^{-t|\xi|^{2\sigma}}\frac{|\xi|^{a+2(j+1)\sigma}}{1-|\xi|^{2\sigma}}\chi(|\xi|)\big(\widehat{u_0}(|\xi|)+\widehat{u_1}(|\xi|)\big)\Big)\Big\|_{L^q}. \label{pro2.4.2}
\end{align}
\normalsize
Applying Young's convolution inequality and Lemma \ref{lemma2.4} we can proceed (\ref{pro2.4.1}) as follows:
\begin{equation}
e^{-t}\,\Big\|\mathcal{F}^{-1}\Big(\frac{|\xi|^{a+2\sigma}}{1- |\xi|^{2\sigma}}\chi(|\xi|)\widehat{u_0}(|\xi|)\Big)\Big\|_{L^q}\lesssim e^{-t}\,\Big\|\mathcal{F}^{-1}\Big(\frac{|\xi|^{a+2\sigma}}{1- |\xi|^{2\sigma}}\chi(|\xi|)\Big)\Big\|_{L^{r_1}}\, \|u_0\|_{L^p}\lesssim e^{-t} \|u_0\|_{L^p},
\label{pro2.4.3}
\end{equation}
where $r_1 \in [1,\ity]$ fulfills $\frac{1}{r_1}+\frac{1}{p}= 1+\frac{1}{q}$, and
\begin{align}
e^{-t}\,\Big\|\mathcal{F}^{-1}\Big(\frac{|\xi|^a}{1- |\xi|^{2\sigma}}\chi(|\xi|)\widehat{u_1}(|\xi|)\Big)\Big\|_{L^q}
&\lesssim \begin{cases}
e^{-t}\,\Big\|\mathcal{F}^{-1}\Big(\frac{|\xi|^a}{1- |\xi|^{2\sigma}}\chi(|\xi|)\Big)\Big\|_{L^{r_1}}\, \|u_1\|_{L^{p}} &\text{ if }\,\,\, a>0 \\
e^{-t}\,\Big\|\mathcal{F}^{-1}\Big(\frac{|\xi|^{\e}}{1- |\xi|^{2\sigma}}\chi(|\xi|)\Big)\Big\|_{L^{r_2}}\, \big\||\xi|^{-\e} u_1\big\|_{L^{p^*}} &\text{ if }\,\,\, a=0
\end{cases} \nonumber \\
&\lesssim e^{-t} \|u_1\|_{L^p}, \label{pro2.4.3}
\end{align}
where $\e$ is a sufficiently small positive constant and $r_2 \in [1,\ity]$ satisfies $1+\frac{1}{q}= \frac{1}{r_2}+ \frac{1}{p^*}$. Here we used the property of the normalized Riez potential in Remark \ref{remark2.3.1} below. In order to control (\ref{pro2.4.2}), we re-write
$$ \frac{1}{1- |\xi|^{2\sigma}}= \sum_{k=0}^\ity |\xi|^{2k\sigma} $$
due to the small value of $|\xi|$. Hence, using Lemma \ref{lemma2.3} we arrive at the following estimate:
\begin{align}
&\Big\|\mathcal{F}^{-1}\Big(e^{-t|\xi|^{2\sigma}}\frac{|\xi|^{a+2(j+1)\sigma}}{1-|\xi|^{2\sigma}}\chi(|\xi|)\big(\widehat{u_0}(|\xi|)+\widehat{u_1}(|\xi|)\big)\Big)\Big\|_{L^q} \nonumber \\ 
&\qquad \lesssim \sum_{k=0}^\ity \Big\|\mathcal{F}^{-1}\Big(|\xi|^{a+2(k+j+1)\sigma} e^{-t|\xi|^{2\sigma}} \chi(|\xi|)\big(\widehat{u_0}(|\xi|)+\widehat{u_1}(|\xi|)\big)\Big)\Big\|_{L^q} \nonumber \\
&\qquad \lesssim \sum_{k=0}^\ity (1+t)^{-\frac{n}{2\sigma}(\frac{1}{p}- \frac{1}{q})- \frac{a}{2\sigma}-(k+j+1)}\|u_0+ u_1\|_{L^p}\lesssim (1+t)^{-\frac{n}{2\sigma}(\frac{1}{p}- \frac{1}{q})- \frac{a}{2\sigma}-j-1}\|u_0+ u_1\|_{L^p} \label{pro2.4.4}
\end{align}
for large $t$. Therefore, from (\ref{pro2.4.1}) to (\ref{pro2.4.4}) we may conclude the desired estimates. This completes the proof of Theorem \ref{theorem2.1}.
\end{proof}

\begin{nx}
\fontshape{n}
\selectfont
Here we want to underline that in the proof of (\ref{pro2.4.3}) we used the property of the normalized Riez potential (see more Remark 2.1 in \cite{DabbiccoEbert2014})
$$ I_\e f(x):= \mathcal{F}^{-1}\big(|\xi|^{-\e}\hat{f}(\xi)\big)= C_{n,\e}\int_{\R^n}\frac{f(y)}{|x-y|^{n-\e}}dy, $$
where $\e \in (0,n)$. In particular, if $f \in L^p$ for some $p \in \big(1,\frac{n}{\e}\big)$, then the following properties hold:
$$ I_\e f \in L^{p^*} \quad \text{ and } \quad \|I_\e f\|_{L^{p^*}} \lesssim \|f\|_{L^p}, \quad \text{ where }\quad \frac{1}{p}- \frac{1}{p^*}= \frac{\e}{n}. $$
\label{remark2.3.1}
\end{nx}

\section{Treatment of the corresponding semi-linear model} \label{Semi-linear estimates}
\subsection{Global (in time) existence of the solution} \label{Global existence}
\begin{proof}[Proof of Theorem \ref{theorem3.1}]
We choose introduce the solution space
$$X(t):= C([0,t],H^{\sigma}) \cap C^1([0,t],L^2), $$
with the norm
\begin{align*}
\|u\|_{X(t)}:= \sup_{0\le \tau \le t} \Big( (1+\tau)^{\frac{n}{2\sigma}(\frac{1}{m}-\frac{1}{2})}\|u(\tau,\cdot)\|_{L^2} &+ (1+\tau)^{\frac{n}{2\sigma}(\frac{1}{m}-\frac{1}{2})+ \frac{1}{2}}\big\||D|^{\sigma} u(\tau,\cdot)\big\|_{L^2} \\ 
&+ (1+\tau)^{\frac{n}{2\sigma}(\frac{1}{m}-\frac{1}{2})+ 1}\|u_t(\tau,\cdot)\|_{L^2}\Big).
\end{align*}
By recalling the fundamental solutions from (\ref{equation2.5}), we can write the solutions of the corresponding linear Cauchy problems with vanishing right-hand sides to (\ref{equation1.1}) as follows:
$$ u^{ln}(t,x)= K_0(t,x) \ast_{x} u_0(x)+ K_1(t,x) \ast_{x} u_1(x), $$
where
$$ K_0(t,x):= \mathcal{F}^{-1}\Big(\frac{e^{-t|\xi|^{2\sigma}}-|\xi|^{2\sigma} e^{-t}}{1-|\xi|^{2\sigma}}\Big) \quad \text{ and }\quad K_1(t,x):= \mathcal{F}^{-1}\Big(\frac{e^{-t|\xi|^{2\sigma}}-e^{-t}}{1-|\xi|^{2\sigma}}\Big). $$
Using Duhamel's principle we get the formal implicit representation of the solutions to (\ref{equation1.1}) in the following form:
$$ u(t,x)= u^{ln}(t,x) + \int_0^t K_1(t-\tau,x) \ast_x |u(\tau,x)|^p d\tau=: u^{ln}(t,x)+ u^{nl}(t,x). $$
We define for all $t>0$ the operator $N: \, u \in X(t) \longrightarrow Nu \in X(t)$ by
$$Nu(t,x)= u^{ln}(t,x)+ u^{nl}(t,x). $$
We will show that the operator $N$ fulfills the following two inequalities:
\begin{align}
\|Nu\|_{X(t)} &\lesssim \|(u_0,u_1)\|_{\mathcal{A}^{\sigma}_m}+ \|u\|^p_{X(t)}, \label{the3.1.1}\\
\|Nu- Nv\|_{X(t)} &\lesssim \|u- v\|_{X(t)} \big(\|u\|^{p-1}_{X(t)}+ \|v\|^{p-1}_{X(t)}\big). \label{the3.1.2}
\end{align}
After that, applying Banach's fixed point theorem we gain local (in time) existence results of large data solutions and global (in time) existence results of small data solutions as well. \\
From the definition of the norm in $X(t)$, by plugging $a=0$, $a=\sigma$ and $j=0,\,1$ into the statements from Proposition \ref{proposition2.1} we arrive at
\begin{equation}
\big\|u^{ln}\big\|_{X(t)} \lesssim \|(u_0,u_1)\|_{\mathcal{A}^{\sigma}_{m}}. \label{the3.1.3}
\end{equation}
Thus, it is reasonable to indicate the following inequality instead of (\ref{the3.1.1}):
\begin{equation}
\|u^{nl}\|_{X(t)} \lesssim \|u\|^p_{X(t)}. \label{the3.1.4}
\end{equation}
At the first stage, let us prove the inequality (\ref{the3.1.4}). In order to deal with some estimates for $u^{nl}$, we use the $(L^m \cap L^2)- L^2$ estimates if $\tau \in [0,t/2]$ and the $L^2-L^2$ estimates if $\tau \in [t/2,t]$ from Proposition \ref{proposition2.1}. As a result, we derive the following estimates for $k=0,\,1$:
\begin{align*}
\big\||D|^{k\sigma} u^{nl}(t,\cdot)\big\|_{L^2} &\lesssim \int_0^{t/2}(1+t-\tau)^{-\frac{n}{2\sigma}(\frac{1}{m}-\frac{1}{2})- \frac{k}{2}}\big\||u(\tau,\cdot)|^p\big\|_{L^m \cap L^2}d\tau\\
&\qquad + \int_{t/2}^t (1+t-\tau)^{-\frac{k}{2}}\big\||u(\tau,\cdot)|^p\big\|_{L^2}d\tau.
\end{align*}
Hence, it is clear to estimate $|u(\tau,x)|^p$ in $L^m \cap L^2$ and $L^2$. Namely, we can proceed as follows:
$$\big\||u(\tau,\cdot)|^p\big\|_{L^m \cap L^2} \lesssim \|u(\tau,\cdot)\|^p_{L^{mp}}+ \|u(\tau,\cdot)\|^p_{L^{2p}} \quad \text{ and } \quad \big\||u(\tau,\cdot)|^p\big\|_{L^2}= \|u(\tau,\cdot)\|^p_{L^{2p}}.$$
By applying the fractional Gagliardo-Nirenberg inequality from Proposition \ref{fractionalGagliardoNirenberg} we obtain
\begin{align}
\big\||u(\tau,\cdot)|^p\big\|_{L^m \cap L^2} &\lesssim (1+\tau)^{-\frac{n}{2m\sigma}(p-1)}\|u\|^p_{X(\tau)}, \label{the3.1.5} \\
\big\||u(\tau,\cdot)|^p\big\|_{L^2} &\lesssim (1+\tau)^{-\frac{n}{2m\sigma}(p-\frac{m}{2})}\|u\|^p_{X(\tau)}, \label{the3.1.6}
\end{align}
provided that the conditions (\ref{GN1A1}) and (\ref{GN1A2}) are satisfied. Consequently, we arrive at
\begin{align*}
\big\||D|^{k\sigma} u^{nl}(t,\cdot)\big\|_{L^2} &\lesssim (1+t)^{-\frac{n}{2\sigma}(\frac{1}{m}-\frac{1}{2})- \frac{k}{2}}\|u\|^p_{X(t)} \int_0^{t/2}(1+\tau)^{-\frac{n}{2m\sigma}(p-1)} d\tau\\
&\qquad + (1+t)^{-\frac{n}{2m\sigma}(p-\frac{m}{2})}\|u\|^p_{X(t)} \int_{t/2}^t (1+t-\tau)^{-\frac{k}{2}}d\tau.
\end{align*}
Here we notice that $(1+t-\tau) \approx (1+t)$ if $\tau \in [0,t/2]$ and $(1+\tau) \approx (1+t)$ if $\tau \in [t/2,t]$. Since the condition (\ref{exponent1A}) holds, it is equivalent to $-\frac{n}{2m\sigma}(p-1)< -1$. Moreover, it is clear to see that $-\frac{k}{2}>-1$ for $k=0,\,1$. Therefore, we get
$$ (1+t)^{-\frac{n}{2\sigma}(\frac{1}{m}-\frac{1}{2})- \frac{k}{2}}\int_0^{t/2}(1+\tau)^{-\frac{n}{2m\sigma}(p-1)} d\tau\lesssim (1+t)^{-\frac{n}{2\sigma}(\frac{1}{m}-\frac{1}{2})- \frac{k}{2}}, $$
and
\begin{equation}
(1+t)^{-\frac{n}{2m\sigma}(p-\frac{m}{2})} \int_{t/2}^t (1+t-\tau)^{-\frac{k}{2}}d\tau \lesssim (1+t)^{-\frac{n}{2m\sigma}(p-\frac{m}{2})+1- \frac{k}{2}}\lesssim (1+t)^{-\frac{n}{2\sigma}(\frac{1}{m}-\frac{1}{2})- \frac{k}{2}}. \label{the3.1.7}
\end{equation}
From both the above estimates we may conclude
$$\big\||D|^{k\sigma} u^{nl}(t,\cdot)\big\|_{L^2} \lesssim (1+t)^{-\frac{n}{2\sigma}(\frac{1}{m}-\frac{1}{2})- \frac{k}{2}}\|u\|^p_{X(t)}, $$
for $k=0,\,1$. In the similar way we also derive the following estimate:
\begin{align*}
\big\|\partial_t u^{nl}(t,\cdot)\big\|_{L^2} &\lesssim (1+t)^{-\frac{n}{2\sigma}(\frac{1}{m}-\frac{1}{2})- 1}\|u\|^p_{X(t)} \int_0^{t/2}(1+\tau)^{-\frac{n}{2m\sigma}(p-1)} d\tau\\
&\qquad + (1+t)^{-\frac{n}{2m\sigma}(p-\frac{m}{2})}\|u\|^p_{X(t)} \int_{t/2}^t (1+t-\tau)^{-1}d\tau.
\end{align*}
It is obvious that the first integral will be handled as before. For this reason, we only need to estimate the second one. In particular, we have
\begin{align}
(1+t)^{-\frac{n}{2m\sigma}(p-\frac{m}{2})} \int_{t/2}^t (1+t-\tau)^{-1}d\tau &\lesssim (1+t)^{-\frac{n}{2m\sigma}(p-\frac{m}{2})}\log(1+t) \nonumber \\ 
&\lesssim (1+t)^{-\frac{n}{2m\sigma}(p-\frac{m}{2})+\e} \lesssim (1+t)^{-\frac{n}{2\sigma}(\frac{1}{m}-\frac{1}{2})- 1}. \label{the3.1.8}
\end{align}
Here because of $-\frac{n}{2m\sigma}(p-1)< -1$, we choose a sufficiently small positive number $\e$ satisfying $\e< -1+\frac{n}{2m\sigma}(p-1)$. Therefore, we arrive at
$$ \big\|\partial_t u^{nl}(t,\cdot)\big\|_{L^2} \lesssim (1+t)^{-\frac{n}{2\sigma}(\frac{1}{m}-\frac{1}{2})- 1}\|u\|^p_{X(t)}. $$
From the definition of the norm in $X(t)$, we may conclude immediately the inequality (\ref{the3.1.4}). \medskip

\noindent Next, let us prove the inequality (\ref{the3.1.2}). Taking account of two elements $u$ and $v$ from $X(t)$, we get
$$Nu(t,x)- Nv(t,x)= u^{nl}(t,x)- v^{nl}(t,x). $$
Using again the $(L^m \cap L^2)- L^2$ estimates if $\tau \in [0,t/2]$ and the $L^2-L^2$ estimates if $\tau \in [t/2,t]$ from Proposition \ref{proposition2.1}, we obtain the following estimates:
\begin{align*}
\big\||D|^{k\sigma} (u^{nl}- v^{nl})(t,\cdot)\big\|_{L^2} &\lesssim \int_0^{t/2}(1+t-\tau)^{-\frac{n}{2\sigma}(\frac{1}{m}-\frac{1}{2})- \frac{k}{2}}\big\||u(\tau,\cdot)|^p- v(\tau,\cdot)|^p\big\|_{L^m \cap L^2}d\tau\\
&\qquad+ \int_{t/2}^t (1+t-\tau)^{-\frac{k}{2}}\big\||u(\tau,\cdot)|^p- v(\tau,\cdot)|^p\big\|_{L^2}d\tau,
\end{align*}
and
\begin{align*}
\big\|\partial_t (u^{nl}- v^{nl})(t,\cdot)\big\|_{L^2} &\lesssim \int_0^{t/2}(1+t-\tau)^{-\frac{n}{2\sigma}(\frac{1}{m}-\frac{1}{2})- 1}\big\||u(\tau,\cdot)|^p- v(\tau,\cdot)|^p\big\|_{L^m \cap L^2}d\tau\\
&\qquad + \int_{t/2}^t (1+t-\tau)^{-1}\big\||u(\tau,\cdot)|^p- v(\tau,\cdot)|^p\big\|_{L^2}d\tau.
\end{align*}
Applying H\"{o}lder's inequality leads to
\begin{align*}
\big\||u(\tau,\cdot)|^p- |v(\tau,\cdot)|^p\big\|_{L^2}& \lesssim \|u(\tau,\cdot)- v(\tau,\cdot)\|_{L^{2p}} \big(\|u(\tau,\cdot)\|^{p-1}_{L^{2p}}+ \|v(\tau,\cdot)\|^{p-1}_{L^{2p}}\big),\\
\big\||u(\tau,\cdot)|^p- |v(\tau,\cdot)|^p\big\|_{L^m}& \lesssim \|u(\tau,\cdot)- v(\tau,\cdot)\|_{L^{mp}} \big(\|u(\tau,\cdot)\|^{p-1}_{L^{mp}}+ \|v(\tau,\cdot)\|^{p-1}_{L^{mp}}\big).
\end{align*}
Analogously to the proof of (\ref{the3.1.4}), we employ the fractional Gagliardo-Nirenberg inequality from Proposition \ref{fractionalGagliardoNirenberg} to the terms
$$ \|u(\tau,\cdot)- v(\tau,\cdot)\|_{L^{\eta}},\quad \|u(\tau,\cdot)\|_{L^{\eta}},\quad \|v(\tau,\cdot)\|_{L^{\eta}} $$
with $\eta=2p$ or $\eta=mp$ to conclude the inequality (\ref{the3.1.2}). Summarizing, the proof of Theorem \ref{theorem3.1} is completed.
\end{proof}

\subsection{Large time behavior of the global solution} \label{Large time behavior}
In order to prove Theorem \ref{theorem3.2}, we need some auxiliary estimates as follows:

\bmd \label{proposition2.2.2}
The Sobolev solutions to (\ref{equation1.2}) satisfy the following estimate for $j,\,k=0,\,1$ and $(j,k)\neq (1,1)$:
$$ \big\|\partial_t^j |D|^{k\sigma} \big(u(t,\cdot)- (P_0+P_1)G_\sigma(t,\cdot)\big)\big\|_{L^2} \lesssim (1+t)^{-\frac{n}{4\sigma}-\frac{k}{2}- j- \frac{1}{2\sigma}} \big(\|u_0\|_{L^{1,1} \cap H^\sigma}+ \|u_1\|_{L^{1,1} \cap L^2}\big), $$
for large $t \ge 1$ and for all space dimensions $n\ge 1$.
\emd

\begin{proof}
Following the proof of Corollary \ref{corollary2.2.1} with a minor modification we can conclude the proof of Proposition \ref{proposition2.2.2}.
\end{proof}

\bmd \label{proposition2.2.3}
The Sobolev solutions to (\ref{equation1.2}) satisfy the following estimate for large $t \ge 1$:
\begin{equation}
\big\|\partial_t^j |D|^a \big(u(t,\cdot)- v(t,\cdot)\big)\big\|_{L^2} \lesssim t^{-\frac{n}{4\sigma}-\frac{a}{2\sigma}-j-1} \|(u_0,u_1)\|_{L^1}+ e^{-t} \big(\|u_0\|_{H^a}+ \|u_1\|_{H^{[a-2\sigma]^+}}\big),
\label{pro2.2.3.2}
\end{equation}
for any $a\ge 0$, $j=0,1$ and for all space dimensions $n\ge 1$.
\emd

\begin{proof}
For small frequencies, we can repeat exactly the same way as we did in the proof of Theorem \ref{theorem2.1} to derive
$$ \big\|\partial_t^j |D|^a \big(u_{\text{\fontshape{n}\selectfont low}}(t,\cdot)- v_{\text{\fontshape{n}\selectfont low}}(t,\cdot)\big)\big\|_{L^2} \lesssim (1+t)^{-\frac{n}{4\sigma}- \frac{a}{2\sigma}-j-1}\|(u_0,u_1)\|_{L^1}+ e^{-t}\|(u_0,u_1)\|_{L^2}. $$
Taking account of large frequencies we can proceed as follows:
\begin{align}
&\big\|\partial_t^j |D|^a \big(u_{\text{\fontshape{n}\selectfont high}}(t,\cdot)- v_{\text{\fontshape{n}\selectfont high}}(t,\cdot)\big)\big\|_{L^2} \nonumber \\
&\qquad \lesssim e^{-t}\,\Big\|\frac{|\xi|^{a+2\sigma}}{|\xi|^{2\sigma}-1}\big(1-\chi(|\xi|)\big)\widehat{u_0}(|\xi|)+ \frac{|\xi|^a}{|\xi|^{2\sigma}-1}\big(1-\chi(|\xi|)\big)\widehat{u_1}(|\xi|)\Big\|_{L^2} \nonumber \\
&\qquad \quad + \Big\|e^{-t|\xi|^{2\sigma}}\frac{|\xi|^{a+2(j+1)\sigma}}{|\xi|^{2\sigma}-1}\big(1-\chi(|\xi|)\big)\big(\widehat{u_0}(|\xi|)+\widehat{u_1}(|\xi|)\big)\Big\|_{L^2} \label{pro2.2.3.3} \\
&\qquad \lesssim e^{-t}\,\big\||\xi|^a\widehat{u_0}(|\xi|)+ |\xi|^{[a- 2\sigma]^+}\widehat{u_1}(|\xi|)\big\|_{L^2} \nonumber \\
&\qquad \quad + \big\|e^{-t|\xi|^{2\sigma}}|\xi|^{a+2(j+1)\sigma}\big(1-\chi(|\xi|)\big)\big\|_{L^2}\|\widehat{u_0}+\widehat{u_1}\|_{L^\ity} \label{pro2.2.3.4} \\
&\qquad \lesssim e^{-t}\,\big(\|u_0\|_{H^a}+ \|u_1\|_{H^{[a- 2\sigma]^+}}\big) + t^{-\frac{n}{4\sigma}-\frac{a}{2\sigma}-j-1} \|(u_0,u_1)\|_{L^1} \label{pro2.2.3.5}
\end{align}
Here we notice that we used Parseval-Plancherel formula in (\ref{pro2.2.3.3}). Moreover, for the second term in (\ref{pro2.2.3.4}) and (\ref{pro2.2.3.5}) we applied H\"{o}lder's inequality and Lemma \ref{LemmaL1normEstimate} combined with Young-Hausdorff inequality, respectively. Therefore, from the above estimates we may conclude the desired statement what we wanted to prove.
\end{proof}

\bmd \label{proposition2.1.2}
The Sobolev solutions to (\ref{equation1.2}) satisfy the following estimate for large $t \ge 1$: 
\begin{equation}
\big\|\partial_t^j |D|^a v(t,\cdot)\big\|_{L^2} \lesssim t^{-\frac{n}{4\sigma}- \frac{a}{2\sigma}-j} \|(u_0,u_1)\|_{L^1},
\label{pro2.1.2.2}
\end{equation}
for any $a\ge 0$, $j=0,1$ and for all space dimensions $n\ge 1$. Moreover, the following estimate holds for any $t>0$:
\begin{equation}
\big\|\partial_t^j |D|^a G_\sigma(t,\cdot)\big\|_{L^2} \lesssim t^{-\frac{n}{4\sigma}- \frac{a}{2\sigma}-j},
\label{pro2.1.2.3}
\end{equation}
for any $a\ge 0$, $j=0,1$ and for all space dimensions $n\ge 1$.
\emd

\begin{proof}
To derive (\ref{pro2.1.2.2}), we repeat some arguments as we did in the proof of the second term in (\ref{pro2.2.3.5}). By the aid of Parseval-Plancherel formula and a change of variables when needed, we may conclude the proof of (\ref{pro2.1.2.3}). Hence, this completes the proof of Proposition \ref{proposition2.1.2}.
\end{proof}

\begin{proof}[Proof of Theorem \ref{theorem3.2}]
Thanks to the statement in Proposition \ref{proposition2.2.2}, we only indicate the following estimate in place of (\ref{LargetimeEstimate3.2}):
\begin{align*}
\Big\|\partial_t^j |D|^{k\sigma}\Big(\int_0^t K_1(t-\tau,x) \ast_x |u(\tau,x)|^p d\tau- &\Big(\int_0^\ity \int_{\R^n} |u(,\tau,y)|^p dyd\tau\Big)G_\sigma(t,x)\Big)\Big\|_{L^2} \\ 
&= o\big(t^{-\frac{n}{4\sigma}- \frac{k}{2}-j}\big),
\end{align*}
by recalling the presentation of the solutions $u(t,x)= u^{ln}(t,x)+u^{nl}(t,x)$ to (\ref{equation1.1}) as in Theorem \ref{theorem3.1}. Due to the fact that $K_1(0,x) \ast_x |u(t,x)|^p= 0$, we can re-write the above estimate in the equivalent form
\begin{align}
\Big\|\int_0^t \partial_t^j |D|^{k\sigma}\big(K_1(t-\tau,x) \ast_x |u(\tau,x)|^p\big) d\tau- &\Big(\int_0^\ity \int_{\R^n} |u(\tau,y)|^p dyd\tau\Big)\partial_t^j |D|^{k\sigma}G_\sigma(t,x)\Big\|_{L^2} \nonumber \\ 
&= o\big(t^{-\frac{n}{4\sigma}- \frac{k}{2}-j}\big). \label{the3.2.1}
\end{align}
Now we shall separate the left-hand side term of (\ref{the3.2.1}) in $L^2$ norm into five sub-terms as follows:
\begin{align}
&\int_0^t \partial_t^j |D|^{k\sigma}\big(K_1(t-\tau,x) \ast_x |u(\tau,x)|^p\big) d\tau- \Big(\int_0^\ity \int_{\R^n} |u(\tau,y)|^p dyd\tau\Big)\partial_t^j |D|^{k\sigma}G_\sigma(t,x) \nonumber \\ 
&\qquad= \int_0^{t/2} \partial_t^j |D|^{k\sigma}\Big(\big(K_1(t-\tau,x)- G_\sigma(t-\tau,\cdot)\big) \ast_x |u(\tau,x)|^p\Big) d\tau \nonumber \\
&\qquad \quad + \int_{t/2}^t \partial_t^j |D|^{k\sigma}\big(K_1(t-\tau,x) \ast_x |u(\tau,x)|^p\big) d\tau \nonumber \\
&\qquad \quad + \int_0^{t/2} \partial_t^j |D|^{k\sigma}\Big(\big(G_\sigma(t-\tau,x)- G_\sigma(t,x) \big)\ast_x |u(\tau,x)|^p\Big) d\tau \nonumber \\
&\qquad \quad + \int_0^{t/2} \partial_t^j |D|^{k\sigma}\Big(G_\sigma(t,x)\ast_x |u(\tau,x)|^p - \Big(\int_{\R^n} |u(\tau,y)|^p dy\Big)G_\sigma(t,x) \Big)d\tau \nonumber \\
&\qquad \quad - \Big(\int_{t/2}^\ity \int_{\R^n} |u(\tau,y)|^p dyd\tau\Big)\partial_t^j |D|^{k\sigma}G_\sigma(t,x) \nonumber \\
&\qquad=: I_1(t,x)+ I_2(t,x)+ I_3(t,x)+ I_4(t,x)- I_5(t,x). \label{the3.2.2}
\end{align}
At first, let us estimate $I_1(t,x)$. Namely, applying the statement (\ref{pro2.2.3.2}) leads to
\begin{align}
\|I_1(t,\cdot)\|_{L^2}& \lesssim \int_0^{t/2} \Big\|\partial_t^j |D|^{k\sigma}\Big(\big(K_1(t-\tau,x)- G_\sigma(t-\tau,\cdot)\big) \ast_x |u(\tau,x)|^p\Big)\Big\|_{L^2} d\tau \nonumber \\
&\lesssim \int_0^{t/2} (t-\tau)^{-\frac{n}{4\sigma}-\frac{k}{2}-j-1} \big\||u(\tau,\cdot)|^p\big\|_{L^1}d\tau+ \int_0^{t/2} e^{-(t-\tau)} \big\||u(\tau,\cdot)|^p\big\|_{L^2}d\tau \nonumber \\
&\lesssim t^{-\frac{n}{4\sigma}-\frac{k}{2}-j-1} \int_0^{t/2} (1+\tau)^{-\frac{n}{2\sigma}(p-1)}d\tau+ e^{-t/2} \int_0^{t/2} (1+\tau)^{-\frac{n}{2\sigma}(p-\frac{1}{2})}d\tau \label{the3.2.3} \\
&\lesssim t^{-\frac{n}{4\sigma}-\frac{k}{2}-j-1}+ e^{-t/2}\lesssim t^{-\frac{n}{4\sigma}-\frac{k}{2}-j-1}. \label{the3.2.4}
\end{align}
Here we used (\ref{the3.1.5}), (\ref{the3.1.6}) and the relation $t-\tau \approx t$ if $\tau \in [0,t/2]$ in (\ref{the3.2.3}). Moreover, in order to derive (\ref{the3.2.4}) we notice that the condition (\ref{exponent1A}) implies the integrability of both the above integrals. In the second step, taking account of $I_2(t,x)$ we repeat exactly the arguments as we did in the proofs of (\ref{the3.1.7}) and (\ref{the3.1.8}) to obtain
\begin{equation}
\|I_2(t,\cdot)\|_{L^2}\lesssim (1+t)^{-\frac{n}{4\sigma}- \frac{k}{2}- j-\e}
\label{the3.2.5}
\end{equation}
where $\e$ is a sufficiently small positive satisfying $2\e< -1+\frac{n}{2\sigma}(p-1)$. To control $I_3(t,x)$, by using the mean value theorem on $t$ we get the following representation:
$$ G_\sigma(t-\tau,x)- G_\sigma(t,x)= -\tau\,\partial_t G_\sigma(t- \omega_1 \tau,x) $$
with a constant $\omega_1 \in [0,1]$. Hence, we can proceed as follows:
\begin{align*}
\|I_3(t,\cdot)\|_{L^2} &\lesssim \int_0^{t/2} \Big\|\partial_t^j |D|^{k\sigma}\Big(\big(G_\sigma(t-\tau,x)- G_\sigma(t,x) \big)\ast_x |u(\tau,x)|^p\Big)\Big\|_{L^2} d\tau \\
&\lesssim \int_0^{t/2} \tau\, \big\|\partial_t^{j+1} |D|^{k\sigma}\big(G_\sigma(t- \omega_1 \tau,x)\ast_x |u(\tau,x)|^p\big)\big\|_{L^2} d\tau
\end{align*}
Employing (\ref{pro2.1.2.2}) gives
\begin{align*}
\|I_3(t,\cdot)\|_{L^2}&\lesssim \int_0^{t/2} \tau\,(t- \omega_1 \tau)^{-\frac{n}{4\sigma}- \frac{a}{2\sigma}-j-1} \big\||u(\tau,\cdot)|^p\big\|_{L^1} d\tau \\
&\lesssim t^{-\frac{n}{4\sigma}- \frac{a}{2\sigma}-j-1}\int_0^{t/2} \tau\, (1+\tau)^{-\frac{n}{2\sigma}(p-1)} d\tau,
\end{align*}
where we used (\ref{the3.1.5}) and the relation $t- \omega_1 \tau \approx t$ if $\tau \in [0,t/2]$. Thus, we may arrive at
\begin{align}
\|I_3(t,\cdot)\|_{L^2}&\lesssim t^{-\frac{n}{4\sigma}- \frac{a}{2\sigma}-j-1}\int_0^{t/2} (1+\tau)^{-\frac{n}{2\sigma}(p-1)+1} d\tau \nonumber \\
&\lesssim t^{-\frac{n}{4\sigma}- \frac{a}{2\sigma}-j-1}\Big(1+ (1+t)^{-\frac{n}{2\sigma}(p-1)+2}+ \log(1+t)\Big) \nonumber \\
&\lesssim t^{-\frac{n}{4\sigma}- \frac{a}{2\sigma}-j-1}(1+t)^{-\e+1} \lesssim t^{-\frac{n}{4\sigma}- \frac{a}{2\sigma}-j-\e} \label{the3.2.6}
\end{align}
as $t \to \ity$, with a sufficiently small positive number $\e$ such that $-\frac{n}{2\sigma}(p-1)+1< -\e$. Let us now devote to the estimate for $I_4(t,x)$. To do this, we shall divide our attention into two parts. In particular, we write
\begin{align*}
I_4(t,x)&= \int_0^{t/2} \partial_t^j |D|^{k\sigma}\Big(\int_{\R^n} G_\sigma(t,x-y) |u(\tau,y)|^p dy - \int_{\R^n} G_\sigma(t,x) |u(\tau,y)|^p dy\Big)d\tau \\
&= \int_0^{t/2} \int_{|y|\le t^{\frac{1}{4\sigma}}} \partial_t^j |D|^{k\sigma}\big(G_\sigma(t,x-y)- G_\sigma(t,x)\big) |u(\tau,y)|^p dyd\tau \\
&\qquad+ \int_0^{t/2} \int_{|y|\ge t^{\frac{1}{4\sigma}}} \partial_t^j |D|^{k\sigma}\big(G_\sigma(t,x-y)- G_\sigma(t,x)\big) |u(\tau,y)|^p dy d\tau =: I_{41}(t,x)+ I_{42}(t,x).
\end{align*}
For the first integral $I_{41}(t,x)$, using the mean value theorem on $x$ we derive
$$ G_\sigma(t,x-y)- G_\sigma(t,x)= -y\,\partial_x G_\sigma(t,x- \omega_2 y) $$
with a constant $\omega_2 \in [0,1]$. For this reason, we may conclude the following estimate:
\begin{align}
\|I_{41}(t,\cdot)\|_{L^2}&\lesssim \Big\|\int_0^{t/2} \int_{|y|\le t^{\frac{1}{4\sigma}}} (-y)\,\partial_t^j |D|^{k\sigma+ 1} G_\sigma(t,x- \omega_2 y) |u(\tau,y)|^p dyd\tau \Big\|_{L^2} \nonumber \\
&\lesssim \int_0^{t/2} \int_{|y|\le t^{\frac{1}{4\sigma}}} |y|\,\big\|\partial_t^j |D|^{k\sigma+ 1} G_\sigma(t,\cdot) \big\|_{L^2} |u(\tau,y)|^p dyd\tau \nonumber \\
&\lesssim t^{-\frac{n}{4\sigma}- \frac{k}{2}-j- \frac{1}{2\sigma}} \int_0^{t/2} \int_{|y|\le t^{\frac{1}{4\sigma}}} |y|\,|u(\tau,y)|^p dyd\tau \qquad \big(\text{by } (\ref{pro2.1.2.3})\big) \nonumber \\
&\lesssim t^{-\frac{n}{4\sigma}- \frac{k}{2}-j- \frac{1}{2\sigma}}\int_0^{t/2} t^{\frac{1}{4\sigma}} \big\||u(\tau,\cdot)|^p\big\|_{L^1}d\tau \nonumber \\
&\lesssim t^{-\frac{n}{4\sigma}- \frac{k}{2}-j- \frac{1}{4\sigma}}\int_0^{t/2} (1+\tau)^{-\frac{n}{2\sigma}(p-1)}d\tau \qquad \big(\text{by } (\ref{the3.1.5})\big) \nonumber \\
&\lesssim t^{-\frac{n}{4\sigma}- \frac{k}{2}-j- \frac{1}{4\sigma}}. \label{the3.2.7}
\end{align}
In order to deal with the other interesting integral $I_{42}(t,x)$, we notice that by (\ref{the3.1.5}) again it holds
$$ \int_0^\ity \int_{\R^n}|u(\tau,y)|^p dyd\tau= \int_0^\ity \big\||u(\tau,\cdot)|^p\big\|_{L^1}d\tau\lesssim \int_0^\ity (1+\tau)^{-\frac{n}{2\sigma}(p-1)}d\tau \lesssim 1. $$
As a result, this deduces immediately the relation
$$ \lim_{t \to \ity} \int_0^\ity \int_{|y|\ge t^{\frac{1}{4\sigma}}}|u(\tau,y)|^p dyd\tau= 0, $$
that is, there exist a sufficiently small positive $\e$ such that
$$ \int_0^\ity \int_{|y|\ge t^{\frac{1}{4\sigma}}}|u(\tau,y)|^p dyd\tau \lesssim t^{-\e}, $$
as $t \to \ity$. Therefore, by (\ref{pro2.1.2.3}) we can estimate $I_{42}(t,x)$ in the following way:
\begin{align}
\|I_{42}(t,\cdot)\|_{L^2}&\le 2\int_0^{t/2} \int_{|y|\ge t^{\frac{1}{4\sigma}}} \big\|\partial_t^j |D|^{k\sigma}G_\sigma(t,\cdot)\big\|_{L^2} |u(\tau,y)|^p dy d\tau \nonumber \\ 
&\lesssim t^{-\frac{n}{4\sigma}- \frac{k}{2}-j}\int_0^{t/2} \int_{|y|\ge t^{\frac{1}{4\sigma}}}|u(\tau,y)|^p dy d\tau \lesssim t^{-\frac{n}{4\sigma}- \frac{k}{2}-j-\e}. \label{the3.2.8}
\end{align}
From (\ref{the3.2.7}) and (\ref{the3.2.8}), we arrive at
\begin{equation}
\|I_{42}(t,\cdot)\|_{L^2}\lesssim t^{-\frac{n}{4\sigma}- \frac{k}{2}-j-\e}.
\label{the3.2.9}
\end{equation}
Finally, we need to control $I_5(t,x)$ to complete our proof. For this purpose, by (\ref{pro2.1.2.3}) again we have
\begin{align}
\|I_5(t,\cdot)\|_{L^2}&\lesssim \big\|\partial_t^j |D|^{k\sigma}G_\sigma(t,\cdot)\big\|_{L^2} \int_{t/2}^\ity \int_{\R^n} |u(\tau,y)|^p dyd\tau \nonumber \\
&\lesssim t^{-\frac{n}{4\sigma}- \frac{a}{2\sigma}-j} \int_{t/2}^\ity \big\||u(\tau,\cdot)|^p\big\|_{L^1} d\tau \lesssim t^{-\frac{n}{4\sigma}- \frac{a}{2\sigma}-j} \int_{t/2}^\ity (1+\tau)^{-\frac{n}{2\sigma}(p-1)} d\tau \nonumber \\ 
&\lesssim t^{-\frac{n}{4\sigma}- \frac{a}{2\sigma}-j} (1+t)^{-\frac{n}{2\sigma}(p-1)+1} \lesssim t^{-\frac{n}{4\sigma}- \frac{a}{2\sigma}-j-\e} \label{the3.2.10}
\end{align}
as $t \to \ity$ and $\e$ is again chosen as a sufficiently small positive to guarantee $-\frac{n}{2\sigma}(p-1)+1< -\e$. Consequently, combining (\ref{the3.2.2}) to (\ref{the3.2.10}) we may conclude (\ref{the3.2.1}). Summarizing, Theorem \ref{theorem3.2} is proved.
\end{proof}

\section{Blow-up result} \label{Blow-up result}
In this section, our aim is to verify the critical exponent to (\ref{equation1.1}). To state our result, we recall the definition of the weak solution to (\ref{equation1.1}) (see, for instance, \cite{IkehataTakeda2017}).
\begin{dn} \label{defweaksolution}
Let $p>1$ and $T>0$. We say that $u \in L^p_{\text{\fontshape{n}\selectfont loc}}([0,T)\times \R^n)$ is a local weak solution to (\ref{equation1.1}) if for any test function $\phi(t,x) \in \mathcal{C}_0^\ity([0,T)\times \R^n)$ it holds
\begin{align}
&\int_0^T \int_{\R^n}|u(t,x)|^p \phi(t,x)dxdt+ \int_{\R^n}u_0(x)\big(\phi(0,x)+ (-\Delta)^{\sigma}\phi(0,x)- \phi_t(0,x)\big)dx+ \int_{\R^n} u_1(x) \phi(0,x)dx \nonumber \\
&\qquad= \int_0^T \int_{\R^n}u(t,x)\big(\phi_{tt}(t,x)- (-\Delta)^{\sigma}\phi_t(t,x)+ (-\Delta)^{\sigma}\phi(t,x)- \phi_t(t,x) \big)dxdt. \label{ptweaksolution}
\end{align}
\end{dn}
If $T= \ity$, we say that $u$ is a global weak solution to (\ref{equation1.1}).

The main ideas of our proof of blow-up result are based on a contradiction argument by using the test function method (see, for example, \cite{DabbiccoEbert2017,Zhang}). Due to the fact that this method, in general, cannot be directly applied to the fractional Laplacian operators $(-\Delta)^\sigma$ as well-known non-local operators, the assumption for integer $\sigma$ comes into play in our proof.
\begin{proof}[Proof of Theorem \ref{dloptimal4.1}]
At first, let us introduce the test functions $\eta= \eta(t)$ and $\varphi=\varphi(x)$ satisfying the following properties:
\begin{align}
&1.\quad \eta \in \mathcal{C}_0^\ity([0,\ity)) \text{ and }
\eta(t)=\begin{cases}
1 \quad \text{ for }0 \le t \le 1/2, \\
0 \quad \text{ for }t \ge 1,
\end{cases} & \nonumber \\ 
&2.\quad \varphi \in \mathcal{C}_0^\ity(\R^n) \text{ and }
\varphi(x)= \begin{cases}
1 \quad \text{ for } |x|\le 1/2, \\
0 \quad \text{ for }|x|\ge 1,
\end{cases} & \nonumber \\
&3.\quad \eta^{-\frac{p'}{p}}\big(|\eta'|^{p'}+|\eta''|^{p'}\big) \text{ and } \varphi^{-\frac{p'}{p}} |\Delta^{\sigma}\varphi|^{p'} \text{ are bounded, } & \label{condition3}
\end{align}
where $p'$ is the conjugate of $p$. In addtion, we suppose that $\eta(t)$ is a decreasing function and that $\varphi=\varphi(|x|)$ is a radial function fulfilling $\varphi(|x|) \le \varphi(|y|)$ for any $|x|\ge |y|$.\medskip

\noindent Let $R$ be a large parameter in $[0,\ity)$. We define the following test function:
$$ \phi_R(t,x):= \eta_R(t) \varphi_R(x), $$
where $\eta_R(t):= \eta(R^{-2\sigma}t)$ and $\varphi_R(x):= \varphi(R^{-1}x)$. Moreover, we define the funtional
$$ I_R:= \int_0^{\ity}\int_{\R^n}|u(t,x)|^p \phi_R(t,x) dxdt= \int_{Q_R}|u(t,x)|^p \phi_R(t,x) d(x,t), $$
where $$Q_R:= [0,R^{2\sigma}] \times B_R,\qquad B_R:= \big\{x\in \R^n: |x|\le R \big\}. $$
Let us now assume that $u= u(t,x)$ is a global (in time) energy solution to (\ref{equation1.1}). Replacing $\phi(t,x)= \phi_R(t,x)$ in (\ref{ptweaksolution}) we arrive at
\begin{align} 
&I_R+ \int_{B_R} u_1(x) \varphi_R(x)dx \nonumber \\
&\,\,\,= \int_{Q_R}u(t,x) \Big(\eta''_R(t) \varphi_R(x)- \eta'_R(t) (-\Delta)^{\sigma}\varphi_R(x)+ \eta_R(t) (-\Delta)^{\sigma}\varphi_R(x)- \eta'_R(t) \varphi_R(x)\Big)d(x,t). \label{the4.1.1}
\end{align}
By applying H\"{o}lder's inequality with $\frac{1}{p}+\frac{1}{p'}=1$ we obtain
\begin{align*}
&\int_{Q_R} |u(t,x)|\, \big|\eta''_R(t) \varphi_R(x)\big| d(x,t) \\
&\qquad \le \Big(\int_{Q_R} \Big|u(t,x)\phi^{\frac{1}{p}}_R(t,x)\Big|^p d(x,t)\Big)^{\frac{1}{p}} \Big(\int_{Q_R} \Big|\phi^{-\frac{1}{p}}_R(t,x) \eta''_R(t) \varphi_R(x)\Big|^{p'} d(x,t)\Big)^{\frac{1}{p'}} \\
&\qquad \le I_R^{\frac{1}{p}} \Big( \int_{Q_R}\eta_R^{-\frac{p'}{p}}(t) \big|\eta''_R(t)\big|^{p'} \varphi_R(x) d(x,t)\Big)^{\frac{1}{p'}}.
\end{align*}
Using the change of variables $\tilde{t}:= R^{-2\sigma}t$ and $\tilde{x}:= R^{-1}x$ gives
\begin{equation} \label{the4.1.2}
\int_{Q_R} |u(t,x)|\, \big|\eta''_R(t) \varphi_R(x)\big| d(x,t) \lesssim I_R^{\frac{1}{p}}\, R^{-4\sigma+ \frac{n+2\sigma}{p'}}. 
\end{equation}
Here we notice that $ \eta''_R(t)= R^{-4\sigma}\eta''(\tilde{t})$ and the assumption (\ref{condition3}) holds. In the analogous treament, we also derive the following estimates:
\begin{align}
\int_{Q_R}|u(t,x)|\, \big|\eta'_R(t) (-\Delta)^{\sigma}\varphi_R(x)\big| d(x,t) &\lesssim I_R^{\frac{1}{p}}\, R^{-4\sigma+ \frac{n+2\sigma}{p'}}, \label{the4.1.3} \\ 
\int_{Q_R}|u(t,x)|\, \big|\eta_R(t) (-\Delta)^{\sigma}\varphi_R(x)\big| d(x,t) &\lesssim I_R^{\frac{1}{p}}\, R^{-2\sigma+ \frac{n+2\sigma}{p'}}, \label{the4.1.4} \\
\int_{Q_R}|u(t,x)|\, \big|\eta'_R(t) \varphi_R(x)\big| d(x,t) &\lesssim I_R^{\frac{1}{p}}\, R^{-2\sigma+ \frac{n+2\sigma}{p'}}, \label{the4.1.5}
\end{align}
where we used
$$ \eta'_R(t)= R^{-2\sigma}\eta'(\tilde{t}) \quad \text{ and }\quad (-\Delta)^{\sigma}\varphi_R(x)= R^{-2\sigma}(-\Delta)^{\sigma}\varphi(\tilde{x}), $$
since $\sigma$ is a integer. Due to the assumption (\ref{optimal1}), there exists a constant $R_0>0$ such that it holds
\begin{equation} \label{the4.1.6}
\int_{B_R} u_1(x) \varphi_R(x)dx >0,
\end{equation}
for any $R>R_0$. As a result, from (\ref{the4.1.1}) to (\ref{the4.1.6}) we may conclude the following estimate:
$$ I_R \lesssim I_R^{\frac{1}{p}}\, R^{-2\sigma+ \frac{n+2\sigma}{p'}}, $$
this is,
\begin{equation} \label{the4.1.7}
I_R^{\frac{1}{p'}} \lesssim R^{-2\sigma+ \frac{n+\alpha}{p'}}.
\end{equation}
It is obvious to see that the assumption (\ref{optimal2}) is re-written in the equivalent form as follows:
$$-2\sigma+ \frac{n+\alpha}{p'} \le 0. $$
Hence, it is reasonable to separate our consideration into two subcases. Taking account of the first subcase $-2\sigma+ \frac{n+\alpha}{p'}< 0$, we let $R \to \ity$ in (\ref{the4.1.7}) to get
$$ \int_0^{\ity}\int_{\R^n}|u(t,x)|^p dxdt= 0, $$
which implies immediately $u \equiv 0$. This is a contradiction to the assumption (\ref{optimal1}). Let us now devote our attention to the second subcase $-2\sigma+ \frac{n+\alpha}{p'}= 0$. By (\ref{the4.1.7}) it follows
$$ I_R= \int_{Q_R}|u(t,x)|^p \phi_R(t,x) d(x,t) \le C, $$
for a sufficiently large $R$ and a suitable positive constant $C$. For this reason, we derive
\begin{equation} \label{the4.1.8}
\int_{\tilde{Q}_R}|u(t,x)|^p \phi_R(t,x) d(x,t) \to 0 \text{ as } R \to \ity,
\end{equation}
where we introduce notations
$$\tilde{Q}_R:= Q_R \setminus \big([0,R^{2\sigma}/2] \times B_{R/2}\big),\qquad B_{R/2}:= \big\{x\in \R^n: 0\le |x|\le R/2 \big\}. $$
Because of $\partial^2_t \phi_R(t,x)= \partial_t\phi_R(t,x)= (-\Delta)^{\sigma}\partial_t\phi_R(t,x)= (-\Delta)^{\sigma}\phi_R(t,x)=0$ in $(\R^1_+ \times \R^n) \setminus \tilde{Q}_R$, we repeat the steps of the proofs from (\ref{the4.1.1}) to (\ref{the4.1.5}) to arrive at the following estimate:
$$ I_R+ \int_{B_R} u_1(x) \varphi_R(x)dx \lesssim \Big(\int_{\tilde{Q}_R}|u(t,x)|^p \phi_R(t,x) d(x,t)\Big)^{\frac{1}{p}}\, R^{-2\sigma+ \frac{n+2\sigma}{p'}}. $$
Due to $-2\sigma+ \frac{n+\alpha}{p'}= 0$, from the above estimate and (\ref{the4.1.6}) we obtain
\begin{equation} \label{the4.1.9}
I_R< I_R+ \int_{B_R} u_1(x) \varphi_R(x)dx \lesssim \Big(\int_{\tilde{Q}_R}|u(t,x)|^p \phi_R(t,x) d(x,t)\Big)^{\frac{1}{p}},
\end{equation}
for any $R>R_0$. By combining (\ref{the4.1.8}) and (\ref{the4.1.9}), we let $R \to \ity$ to conclude
$$ \int_0^{\ity}\int_{\R^n}|u(t,x)|^p dxdt= 0, $$
which is again a contradiction to the assumption (\ref{optimal1}). Summarizing, Theorem \ref{dloptimal4.1} is proved.
\end{proof}

\bnx
\fontshape{n}
\selectfont
Here we want to emphasize that for all small positive constants $\e$ the lifespan $T_\e$ of the solution to the given data $(0,\e u_1)$ in Theorem \ref{dloptimal4.1} can be estimated as follows:
\begin{equation} \label{lifetime}
T_\e \le C\,\e^{-\frac{2\sigma(p-1)}{2\sigma- n(p-1)}} \quad \text{ with }C>0.
\end{equation}
Indeed, let us now consider the case of subcritical exponents. We suppose that $u= u(t,x)$ is a local (in time) energy solution to (\ref{equation1.1}) in $([0,T)\times \R^n)$. To varify the lifespan estimate, we take the initial data $(0,\e u_1)$ in place of $(0,u_1)$ with a small positive constant $\e$ where $u_1\in L_1 \cap L_2$ satisfies the assumption (\ref{optimal1}). In the same way as we did in the steps of the proof of Theorem \ref{dloptimal4.1}, we obtain the following estimte:
\begin{equation}
I_R+ c\e \le C\, I_R^{\frac{1}{p}}\, R^{-2\sigma+ \frac{n+2\sigma}{p'}}.
\label{the4.1.10}
\end{equation}
Here we notice that due to the assumption (\ref{optimal1}), we choose a suitable constant $c$ such that it holds
$$\int_{B_R} u_1(x) \varphi_R(x)dx> c> 0 $$
for any $R> R_0$. From (\ref{the4.1.10}) we arrive at
\begin{equation}
c\e \le C\, I_R^{\frac{1}{p}}\, R^{-2\sigma+ \frac{n+2\sigma}{p'}}- I_R. \label{the4.1.11}
\end{equation}
By the aid of the elementary inequality
$$ A\,y^\gamma- y \le A^{\frac{1}{1-\gamma}} \text{ for any } A>0,\, y \ge 0 \text{ and } 0< \gamma< 1, $$
a straightforward computation gives from (\ref{the4.1.11})
$$ \e \le C\,R^{-2\sigma p'+ n+ 2\sigma}= C\,T^{-\frac{2\sigma p'- n- 2\sigma}{2\sigma}}= C\,T^{-\frac{2\sigma- n(p-1)}{2\sigma(p-1)}}, $$
with $R= T^{\frac{1}{2\sigma}}$. Summarizing, letting $T \to T_\e - 0$ we may conclude (\ref{lifetime}).
\enx


\noindent \textbf{Acknowledgments}\medskip

\noindent The PhD study of the second author is supported by Vietnamese Government's Scholarship (Grant number: 2015/911).\medskip

\noindent\textbf{Appendix A}\medskip

\noindent \textit{A.1. Fractional Gagliardo-Nirenberg inequality}
\begin{md} \label{fractionalGagliardoNirenberg}
Let $1<p,\, p_0,\, p_1<\infty$, $\sigma >0$ and $s\in [0,\sigma)$. Then, it holds the following fractional Gagliardo-Nirenberg inequality for all $u\in L^{p_0} \cap \dot{H}^\sigma_{p_1}$:
$$ \|u\|_{\dot{H}^{s}_p}\lesssim \|u\|_{L^{p_0}}^{1-\theta}\,\, \|u\|_{\dot{H}^{\sigma}_{p_1}}^\theta, $$
where $\theta=\theta_{s,\sigma}(p,p_0,p_1)=\frac{\frac{1}{p_0}-\frac{1}{p}+\frac{s}{n}}{\frac{1}{p_0}-\frac{1}{p_1}+\frac{\sigma}{n}}$ and $\frac{s}{\sigma}\leq \theta\leq 1$ .
\end{md}
For the proof one can see \cite{Ozawa}.
\medskip

\noindent \textit{A.2. Modified Bessel functions}

\begin{md} \label{FourierModifiedBesselfunctions}
Let $f \in L^p(\R^n)$, $p\in [1,2]$, be a radial function. Then, the Fourier transform $F(f)$ is also a radial function and it satisfies
$$ F(f)(\xi)= c \int_0^\ity g(r) r^{n-1} \tilde{J}_{\frac{n}{2}-1}(r|\xi|)dr,\quad g(|x|):= f(x), $$
where $\tilde{J}_\mu(s):=\frac{J_\mu(s)}{s^\mu}$ is called the modified Bessel function with the Bessel function $J_\mu(s)$ and a non-negative integer $\mu$.
\end{md}

\begin{md} \label{PropertiesModifiedBesselfunctions}
The following properties of the modified Bessel function hold:
\begin{enumerate}
\item $s\,d_s\tilde{J}_\mu(s)= \tilde{J}_{\mu-1}(s)-2\mu \tilde{J}_\mu(s)$,
\item $d_s\tilde{J}_\mu(s)= -s\tilde{J}_{\mu+1}(s)$,
\item $\tilde{J}_{-\frac{1}{2}}(s)= \sqrt{\frac{2}{\pi}}\cos s$ and $\tilde{J}_{\frac{1}{2}}(s)= \sqrt{\frac{2}{\pi}} \frac{\sin s}{s}$,
\item $|\tilde{J}_\mu(s)| \le Ce^{\pi|\fontshape{n}\selectfont\text{Im}\mu|} \text{ if } s \le 1, $ \\
and $\tilde{J}_\mu(s)= Cs^{-\frac{1}{2}}\cos \big( s-\frac{\mu}{2}\pi- \frac{\pi}{4} \big) +\mathcal{O}(|s|^{-\frac{3}{2}}) \text{ if } |s|\ge 1$,
\item $\tilde{J}_{\mu+1}(r|x|)= -\frac{1}{r|x|^2}\partial_r \tilde{J}_\mu(r|x|)$, $r \ne 0$, $x \ne 0$.
\end{enumerate}
\end{md}
\medskip

\noindent \textit{A.3. Useful lemmas}
\bbd \label{LemmaL1normEstimate}
Let $n\ge1$, $c>0$, $\alpha >0$ and $\beta \in \R$ satisfy $n+\beta >0$. 
The following estimates hold for $t>0${\rm :}
$$ \int_{|\xi|\le1} |\xi|^\beta e^{-c|\xi|^\alpha t}d\xi \lesssim (1+t)^{-\frac{n+\beta}{\alpha}} \quad \text{ and } \quad \int_{|\xi|\ge1} |\xi|^\beta e^{-c|\xi|^\alpha t}d\xi \lesssim t^{-\frac{n+\beta}{\alpha}}. $$
\ebd
\begin{proof}
In order to prove the first desired estimate, we shall split our consideration into two cases. In the first case $t \in (0,1]$, we get immediately the following estimate:
$$ \int_{|\xi| \le 1} |\xi|^\beta e^{-c|\xi|^\alpha t}d\xi \lesssim \int_0^1 |\xi|^{n+\beta-1} e^{-c|\xi|^\alpha t}d|\xi| \lesssim 1. $$
For the second case $t \in [1,\ity)$, we carry out the change of variables $\xi^\alpha t= \eta^\alpha$, that is, $\xi= t^{-\frac{1}{\alpha}} \eta$ to dervie
$$ \int_{|\xi| \le 1} |\xi|^\beta e^{-c|\xi|^\alpha t}d\xi \lesssim t^{-\frac{n+\beta}{\alpha}} \int_0^\ity |\eta|^{n+\beta-1} e^{-c|\eta|^\alpha}d|\eta| \lesssim t^{-\frac{n+\beta}{\alpha}}. $$
Hence, from the above two estimates we can conclude the desired statement. In the same treatment we can prove the second estimate. This completes our proof.
\end{proof}

\bbd[\cite{IkehataMichihisa}, \cite{Michihisa1}]
Let $n\ge1$ and $f\in L^{1,\gamma}(\R^n)$ with $\gamma\ge0$. 
\begin{enumerate}
\item[{\rm(i)}] 
It holds that 
\begin{align}
\label{moment1}
\biggr|
\hat{f}(\xi)
-\sum_{|\alpha|\le[\gamma]}
M_\alpha(f)(i\xi)^\alpha
\biggr|
\lesssim 
|\xi|^\gamma \|f\|_{L^{1,\gamma}}
\end{align}
for $\xi\in\R^n$. 
\item[{\rm(ii)}] 
It is true that 
\begin{align}
\label{moment2}
\lim_{t\to\infty}
t^{\frac{n}{4\sigma}+\frac{\gamma}{2\sigma}}
\biggr\|
e^{-t|\xi|^{2\sigma}}\hat{f}(\xi)
-\sum_{|\alpha|\le[\gamma]}
M_\alpha(f)(i\xi)^\alpha 
e^{-t|\xi|^{2\sigma}}
\biggr\|_{L^2}
=0.
\end{align}
\end{enumerate}
\ebd
The proof of inequality~\eqref{moment1} was given in \cite{IkehataMichihisa} and \cite{Michihisa1}. 
With a slight modification, we can also obtain \eqref{moment2} and thus the proof is omitted.



\begin{thebibliography}{00}

\bibitem{Dabbicco} M. D'Abbicco, \textit{$L^1-L^1$ estimates for a doubly dissipative semilinear wave equation}, Nonlinear Differ. Equ. Appl., \textbf{24} (2017), 1-23.
\bibitem{DabbiccoEbert2014} M. D'Abbicco, M.R. Ebert, \textit{Diffusion phenomena for the wave equation with structural damping in the $L^p-L^q$ framework}, J. Differ. Equ., \textbf{256} (2014), 2307-2336.
\bibitem{DabbiccoEbert2016} M. D'Abbicco, M.R. Ebert. \textit{A classifiation of structural dissipations for evolution operators}, Math. Methods Appl. Sci., \textbf{39} (2016), 2558-2582.
\bibitem{DabbiccoEbert2017} M. D'Abbicco, M.R. Ebert, \textit{A new phenomenon in the critical exponent for structurally damped semi-linear evolution equations}, Nonlinear Anal., \textbf{149} (2017), 1-40.
\bibitem{DabbiccoReissig} M. D'Abbicco, M. Reissig, Semilinear structural damped waves, \textit{Math. Methods Appl. Sci.}, \textbf{37} (2014), 1570--1592.
\bibitem{DaoReissig1} T.A. Dao, M. Reissig, \textit{An application of $L^1$ estimates for oscillating integrals to parabolic like semi-linear structurally damped $\sigma$-evolution models}, J. Math. Anal. Appl., \textbf{476} (2019), 426--463.
\bibitem{DaoReissig2} T.A. Dao, M. Reissig, \textit{$L^1$ estimates for oscillating integrals and their applications to semi-linear models with $\sigma$-evolution like structural damping}, Discrete Contin. Dyn. Syst. A, \textbf{39} (2019), 5431--5463.
\bibitem{DuongReissig} P. T. Duong, M. Reissig, \textit{The external damping Cauchy problems with general powers of the Laplacian}, in: New Trends in Analysis and Interdisciplinary Applications: Trends in Mathematics, Birkh\"{a}user, Cham (2017), pp. 537-543.
\bibitem{ReissigEbert} M.R. Ebert, M. Reissig,  \textit{Methods for partial differential equations, qualitative properties of solutions, phase space analysis, semilinear models}, Birkh\"{a}user, 2018.
\bibitem{GalaktionovMitidieri} V.A. Galaktionov, E.L. Mitidieri, S.I. Pohozaev, Blow-up for higher-order prabolic, hyperbolic, dispersion and Schr\"{o}dinger equations, in \textit{Monogr. Res. Notes Math.}, Chapman and Hall/CRC, 2014.
\bibitem{Ozawa} H. Hajaiej, L. Molinet, T. Ozawa, B. Wang,  \textit{Necessary and sufficient conditions for the fractional Gagliardo-Nirenberg inequalities and applications to Navier-Stokes and generalized boson equations, Harmonic analysis and nonlinear partial differential equations}, Res.Inst.Math.Sci. (RIMS), RIMS Kokyuroku Bessatsu, B26, Kyoto, (2011), 159-175.
\bibitem{Ikehata} R. Ikehata, \textit{Asymptotic profiles for wave equations with strong damping}, J. Differ. Equ., \textbf{257} (2014), 2159--2177.
\bibitem{IkehataMichihisa} R. Ikehata, H. Michihisa, \textit{Moment conditions and lower bounds in expanding solutions of wave equations with double damping terms}, Asymptot. Anal., (to appear).
\bibitem{IkehataSawada} R. Ikehata, A. Sawada, \textit{Asymptotic profile of solutions for wave equations with frictional and viscoelastic damping terms}, Asymptot. Anal., \textbf{98} (2016), 59-77.
\bibitem{IkehataTakeda2017} R. Ikehata, H. Takeda, \textit{Critical exponent for nonlinear wave equations with frictional and viscoelastic damping terms}, Nonlinear Anal., \textbf{148} (2017), 228-253.
\bibitem{IkehataTodorovaYordanov} R. Ikehata, G. Todorova, and B. Yordanov, \textit{Wave equations with strong damping in Hilbert spaces}, J. Differ. Equ., \textbf{254} (2013), 3352--3368.
\bibitem{Michihisa1} H. Michihisa, \textit{Optimal leading term of solutions to wave equations with strong damping terms}, Hokkaido Math. J., in press.
\bibitem{Michihisa2} H. Michihisa, \textit{Expanding methods for evolution operators of strongly damped wave equations}, (2018), submitted.
\bibitem{Narazaki} T. Narazaki. \textit{$L^p-L^q$ estimates for damped wave equations and their applications to semi-linear problem}, J. Math. Soc. Japan, \textbf{56} (2004), 585--626.
\bibitem{Takeda} H. Takeda, Higher-order expansion of solutions for a damped wave equation, Asymptot. Anal., \textbf{94} (2015), 1-31.
\bibitem{Zhang} Q.S. Zhang, \textit{A blow-up result for a nonlinear wave equation with damping: the critical case}, C. R. Acad. Sci. Paris S\'{e}r. I Math., \textbf{333} (2001), 109-114.

\end{thebibliography}
\end{document}